\documentclass{article}
\usepackage{amssymb}
\usepackage{graphicx}
\usepackage{amsmath}

\newtheorem{theo}{Theorem}
\newtheorem{lem}{Lemma}
\newtheorem{prop}{Proposition}

\newtheorem{rem}{Remark}
\newtheorem{exam}{Example}
\newtheorem{definiti}{Definition}
\newenvironment{proof}[1][Proof]{\noindent \textbf{#1.} }{\ \rule{0.5em}{0.5em}}
\input{tcilatex}
\begin{document}

\author{M. J. C\'{a}novas\thanks{%
Center of Operations Research, Miguel Hern\'{a}ndez University of Elche,
03202 Elche (Alicante), Spain (canovas@umh.es, parra@umh.es).} \and M. A. L%
\'{o}pez\thanks{%
Department of Mathematics, University of Alicante, 03080 Alicante, Spain
(marco.antonio@ua.es); CIAO, Federation University, Ballarat, Australia.
Research of this author is also partially supported by the Australian
Research Council (ARC) Discovery Grants Scheme (Project Grant \#
DP180100602). } \and B. S. Mordukhovich\thanks{
Department of Mathematics, Wayne State University, Detroit, MI 48202, USA
(boris@math.wayne.edu). Research of this author was partially supported by
the USA National Science Foundation under grants DMS-1512846 and
DMS-1808978, by the USA Air Force Office of Scientific Research grant
\#15RT04, and by Australian Research Council under grant DP-190100555.} \and %
J. Parra\footnotemark[2]}
\date{}
\title{Subdifferentials and Stability Analysis\\
of Feasible Set and Pareto Front Mappings\\
in Linear Multiobjective Optimization\thanks{%
This research has been partially supported by grants
MTM2014-59179-C2-(1,2)-P and PGC2018-097960-B-C2(1,2).}}
\maketitle

\begin{abstract}
The paper concerns multiobjective linear optimization problems in $\mathbb{R}%
^{n}$ that are parameterized with respect to the right-hand side
perturbations of inequality constraints. Our focus is on measuring the
variation of the feasible set and the Pareto front mappings around a nominal
element while paying attention to some specific directions. This idea is
formalized by means of the so-called epigraphical multifunction, which is
defined by adding a fixed cone to the images of the original mapping.
Through the epigraphical feasible and Pareto front mappings we describe the
corresponding vector subdifferentials and employ them to verifying
Lipschitzian stability of the perturbed mappings with computing the
associated Lipschitz moduli. The particular case of ordinary linear programs
is analyzed, where we show that the subdifferentials of both multifunctions
are proportional subsets. We also provide a method for computing the optimal
value of linear programs without knowing any optimal solution. Some
illustrative examples are also given in the paper.


\end{abstract}

\section{Introduction and Overview}

\label{intro}

The original motivation for this paper comes from analyzing Lipschitzian
behavior of the so-called Pareto front mapping associated with the \emph{%
multiobjective linear programming} (MLP) problem given by 
\begin{equation}
\begin{tabular}{ccc}
$MLP\left( b\right) :$ & \emph{\textrm{minimize}} & $\left( \left\langle
c_{1},x\right\rangle ,...,\left\langle c_{q},x\right\rangle \right) $ \\ 
& \emph{subject to} & $x\in \mathcal{F}\left( b\right) $,%
\end{tabular}
\label{eq_MOP}
\end{equation}%
where $x\in \mathbb{R}^{n}$ is the decision variable, where $%
c_{1},...,c_{q}\in \mathbb{R}^{n}$ are fixed, and where $\mathcal{F}\left(
b\right) $ is the \emph{feasible set} of the linear inequality system in $%
\mathbb{R}^{n}$ parameterized by its right-hand side (RHS) as 
\begin{equation}
\sigma \left( b\right) :=\big\{\left\langle a_{t},x\right\rangle \leq
b_{t},\,\,\,t\in T:=\left\{ 1,...,m\right\} \big\}  \label{eq_sigma}
\end{equation}%
with the coefficients $a_{t}\in \mathbb{R}^{n}$ fixed for each $t\in T$ and
the perturbation parameter $b=\left( b_{t}\right) _{t\in T}\in \mathbb{R}%
^{T} $ in the RHS of \eqref{eq_sigma}.

For each $b\in \mathbb{R}^{T}$ denote by $\mathcal{S}\left( b\right) $ the
set of \emph{nondominated solutions} to $MLP\left( b\right) $, i.e., $%
\mathcal{S}\left( b\right) $ is formed by all $x\in \mathcal{F}\left(
b\right) $ such that there does not exist any other feasible point $y\in 
\mathcal{F}\left( b\right) $ for which $\left\langle c_{i},y\right\rangle
\leq \left\langle c_{i},x\right\rangle $ whenever $i=1,...,q$ and $%
\left\langle c_{i_{0}},y\right\rangle <\left\langle c_{i_{0}},x\right\rangle 
$ for some $i_{0}\in \{1,...,q\}$. Alternatively it can be reformulated as
follows: considering the mapping $\mathcal{C}\colon \mathbb{R}%
^{n}\rightarrow \mathbb{R}^{q}$ defined by $\mathcal{C}(x):=\left(
\left\langle c_{1},x\right\rangle ,...,\left\langle c_{q},x\right\rangle
\right) $, we have the equivalence 
\begin{equation*}
x\in \mathcal{S}\left( b\right) \ \Leftrightarrow \ \big(\mathcal{C}(%
\mathcal{F}\left( b\right) -x)\big)\cap \left( -\mathbb{R}_{+}^{q}\right)
=\{0_{q}\}.
\end{equation*}

Associated with the parameterized problem (\ref{eq_MOP}), we define the\emph{%
\ Pareto front mapping} $\mathcal{P}\colon\mathbb{R}^{T}\mathbb{%
\rightrightarrows R}^{q}$ by 
\begin{equation}  \label{eq_pareto}
\mathcal{P}\left(b\right):=\left\{\left( \left\langle
c_{1},x\right\rangle,...,\left\langle c_{q},x\right\rangle\right),\text{ }
x\in\mathcal{S}\left(b\right)\right\}=\mathcal{C}(\mathcal{S}\left(b\right)).
\end{equation}
Observe that in the case of ordinary/scalar linear programming (LP) problem,
i.e., \textbf{\ }when $q=1,$ the Pareto front mapping $\mathcal{P}$ reduces
to the real-valued \emph{optimal value function} known also as the `marginal
function' in variational analysis.\vspace*{0.05in}

Appropriate tools of \emph{variational analysis} and \emph{generalized
differentiation} are our primary machinery to study the major (robust) \emph{%
Lipschitzian stability} notion for the feasible set and Pareto front
mappings. To proceed, we need to compute the \emph{subdifferential }of these
set-valued mappings/multifunctions, which is defined via the \emph{%
coderivative} of the corresponding epigraphical multifunctions; see
Section~2. At this moment we advance that a natural definition of the \emph{%
epigraphical Pareto front mapping} $\mathcal{E}_{\mathcal{P}}\colon\mathbb{R}%
^{T}\mathbb{\rightrightarrows R}^{q}$ is given by 
\begin{equation}  \label{eq_epi_Pareto_front}
\mathcal{E}_{\mathcal{P}}\left(b\right):=\mathcal{P}\left(b\right)+\mathbb{R}%
_{+}^{q},
\end{equation}
where $\mathbb{R}_{+}^{q}$ is formed by the elements of $\mathbb{R}^{q}$
with nonnegative components.

Roughly speaking, while analyzing optimality in MLP we are interested only
in that region of the feasible set where optimal/nondominated solutions may
be located. A possible idea to skip the noninteresting regions is to
consider a certain epigraphical mapping associated with the feasible set
mapping. In this way we define the \emph{epigraphical feasible set mapping} $%
\mathcal{E}_{\mathcal{F}}\colon\mathbb{R}^{T}\mathbb{\rightrightarrows R}%
^{n} $ by 
\begin{equation}  \label{eq_epi mapp}
\mathcal{E}_{\mathcal{F}}\left(b\right):=\mathcal{F}\left(b\right)+\left%
\{c_{1},...,c_{q}\right\}^{\circ },
\end{equation}
where $\Omega^{\circ}:=\left\{y\in\mathbb{R}^{n}|\left\langle
y,x\right\rangle\ge 0\;\text{ for all }\;x\in\mathbb{R}^{n}\right\}$ stands
for the (positive) polar cone of the set $\Omega\subset\mathbb{R}^{n}$.%
\vspace*{0.05in}

The \emph{main contributions} of our paper are precise calculations of the
subdifferentials of the set-valued mappings $\mathcal{F}$ and $\mathcal{P}$
with the subsequent usage of them to verify Lipschitzian stability of these
mappings and computing the corresponding Lipschitz moduli by invoking the
powerful machinery of variational analysis. We show below that the
subdifferentials of these multifunctions and their Lipschitz moduli are
closely related as seen in Theorems~\ref{Cor_lip_EP} and \ref{Prop_subdiffF}%
, and the established relationships are particularly clear in the case of
ordinary (single-objective) linear programs; see Proposition~\ref{Prop_sect5}
and Theorem~\ref{Th lipV}.

Given a mapping $\mathcal{M}\colon Y\rightrightarrows Z$ between metric
spaces $Y$ and $Z$ with the graph 
\begin{equation*}
\mathrm{gph}\mathcal{M}:=\big\{(y,z)\in Y\times Z\big|\;z\in \mathcal{M}(y%
\big\}
\end{equation*}%
and with the same notation $d$ for the metrics on $Y$ and $Z$, its
Lipschitzian behavior is analyzed locally around a fixed point $\left( 
\overline{y},\overline{z}\right) \in \mathrm{gph}\mathcal{M}$ while
reflecting the rate of variation of its images with respect to the variation
of the corresponding preimages. Here we focus on the most natural graphical
extension of the classical local Lipschitz continuity to set-valued mappings
that is spread in variational analysis as the
Lipschitz-like/pseudo-Lipschitz/Aubin property. For definiteness let us say
that $\mathcal{M}$ is \emph{Lipschitz-like} around $\left( \overline{y},%
\overline{z}\right) \in \mathrm{gph}\mathcal{M}$ if there exist
neighborhoods $U\subset Y$ and $V\subset Z\,\ $of $\overline{y}$ and $%
\overline{z}$, respectively, and a constant $\ell \geq 0$ such that we have
the linear estimate 
\begin{equation}
d\big(z,\mathcal{M}\left( y^{\prime }\right) \big)\leq \ell \,d\left(
y,y^{\prime }\right) \;\text{ for all }\;y,y^{\prime }\in U\;\text{ and all }%
\;z\in V\cap \mathcal{M}\left( y\right) .  \label{eq_Aubin_prop}
\end{equation}%
Each constant $\ell $ ensuring (\ref{eq_Aubin_prop}) for associated
neighborhoods $U$ and $V\,\ $is called a Lipschitz constant and the infimum
of such Lipschitz constants is called the \emph{Lipschitz modulus}, or the 
\emph{exact Lipschitz bound} of $\mathcal{M}$ around $\left( \overline{y},%
\overline{z}\right) $, and is denoted by $\mathrm{lip}\mathcal{M}\left( 
\overline{y},\overline{z}\right) $. We can easily check that 
\begin{equation}
\mathrm{lip}\mathcal{M}\left( \overline{y},\overline{z}\right) =\underset{%
\QTATOP{y,y^{\prime }\rightarrow \overline{y}}{z\rightarrow \overline{z}%
,z\in \mathcal{M}\left( y\right) }}{\lim \sup }\frac{d\big(z,\mathcal{M}%
\left( y^{\prime }\right) \big)}{d\left( y,y^{\prime }\right) }  \label{mod}
\end{equation}%
under the convention that $0/0:=0$. It has been well recognized in
variational analysis that the Lipschitz-like property \eqref{eq_Aubin_prop}
and its inverse mapping equivalences known as \emph{metric regularity} and 
\emph{linear openness/covering} play a fundamental role in many aspects of
optimization, equilibrium, systems control, and applications; see the
monographs \cite{bz05,DoRo,Ioffe17,KlKu02,Mor06i,Mor18,rw} and the
references therein.

Using the modulus representation \eqref{mod}, we can rephrase that the main
contribution of this paper is to \emph{explicitly compute} the quantities $%
\mathrm{lip}\mathcal{E}_{\mathcal{F}}\left( \overline{b},\overline{x}\right) 
$ and $\mathrm{lip}\mathcal{E}_{\mathcal{P}}\left( \overline{b},\overline{p}%
\right) ,$ with $\left( \overline{b},\overline{x}\right) \in \mathrm{gph}%
\mathcal{E}_{\mathcal{F}}$ and $\left( \overline{b},\overline{p}\right) \in 
\mathrm{gph}\mathcal{E}_{\mathcal{P}}$ respectively, entirely in terms of
the given data of \eqref{eq_MOP} and \eqref{eq_sigma}. Furthermore, we
advance here that the number $\mathrm{lip}\mathcal{E}_{\mathcal{P}}\left( 
\overline{b},\overline{p}\right) $ provides a lower estimate of $\mathrm{lip}%
\mathcal{P}\left( \overline{b},\overline{p}\right) $, and that both
Lipschitz moduli \emph{agree} for ordinary \emph{linear programs} as shown
in Section~5. Having the precise formulas for computing the moduli $\mathrm{%
lip}\mathcal{E}_{\mathcal{F}}\left( \overline{b},\overline{x}\right) $ and $%
\mathrm{lip}\mathcal{E}_{\mathcal{P}}\left( \overline{b},\overline{p}\right) 
$, the \emph{necessary and sufficient conditions} for \emph{Lipschitzian
stability} of the mappings (\ref{eq_epi_Pareto_front}) and (\ref{eq_epi mapp}%
)---in the sense of the validity of the Lipschitz-like property for these
mappings around the reference points---are formulated now as, respectively, 
\begin{equation*}
\mathrm{lip}\mathcal{E}_{\mathcal{F}}\left( \overline{b},\overline{x}\right)
<\infty \;\mbox{ and }\;\mathrm{lip}\mathcal{E}_{\mathcal{P}}\left( 
\overline{b},\overline{p}\right) <\infty .
\end{equation*}

These achievements are largely based on the \emph{subdifferential} notion
for \emph{set-valued mappings} with \emph{ordered values} introduced in \cite%
{BaMo07} (see also \cite{BaMo10,Mor18}) via the coderivatives concept for
mappings and on the \emph{coderivative criterion} for the Lipschitz-like
property of multifunctions established in \cite{Mo93}. The passage from
coderivatives to subdifferentials of ordered mappings was accomplished in 
\cite{BaMo07} via the usage of \emph{epigraphical multifunctions}: the
pattern well understood in variational analysis for the
subdifferential-coderivative relationship concerning scalar
(extended-real-valued) functions; see, e.g., \cite[Vol.~1, p.\ 84]{Mor06i}.

It is worth mentioning that some coderivative analysis of frontier and
efficient solution mapping was provided in \cite{huy-mor-yao08} for problems
of vector optimization with respect to the so-called \emph{generalized order
optimality} (including Pareto efficiency) in infinite-dimensional spaces.
However, neither precise coderivative formulas, nor subdifferential
analysis, nor computations of Lipschitz moduli were obtained in the general
setting of \cite{huy-mor-yao08} in contrast to what is done in this paper.

Furthermore, while confining to the case of ordinary/scalar linear programs
where $\mathcal{P}$ is the optimal value function, the reader is addressed
to \cite{GCPT19} for different formulas concerning Lipschitz moduli in
various parametric frameworks. Note also that Lipschitzian behavior of the
`ordinary' feasible set mapping $\mathcal{F}$ and the computation of its
modulus were derived for more general models of semi-infinite and infinite
programming in \cite{CDLP05} and \cite{clmp09}. Other stability properties
of the feasible set mapping of linear semi-infinite systems were analyzed in 
\cite[Chapter~6]{GoLo98}.\textbf{\ }Lipschitzian behavior of the optimal
set, again in the context of linear programming problems (in fact, in a
continuous convex semi-infinite setting allowing also perturbations of the
objective function) was studied in \cite{CKLP07}, whereas the associated
Lipschitz modulus was computed in \cite{CGP08}.\vspace*{0.05in}

The rest of the paper is organized as follows. In Section~\ref{prelim} we
present the necessary notation, definitions, and results about
coderivatives, subdifferentials, and Lipschitz moduli that are needed later
on. Section~\ref{feas} is devoted to subdifferential analysis and
Lipschitzian stability of the epigraphical feasible set mappings $\mathcal{E}%
_{\mathcal{F}}$ from \eqref{eq_epi mapp}. Specifically, we provide explicit
descriptions of the subdifferential of $\mathcal{F}$ and the Lipschitz
modulus of $\mathcal{E}_{\mathcal{F}}$ at a given point of its graph. In the
subsequent Section~\ref{scal} we develop a constructive procedure for
deriving the representation of such a mapping as the feasible set mapping
associated with new parameterized systems of linear programming. Section~\ref%
{pareto} is focussed on the precise computations of subdifferential and
Lipschitz modulus of the epigraphical Pareto front mapping $\mathcal{E}_{%
\mathcal{P}}$ from \eqref{eq_epi_Pareto_front}. In Section~\ref{lp} we
consider the case of ordinary linear programs (with only one objective
function) and show that even in this case our results are new. In
particular, we establish exact relationships between the subdifferentials
and Lipschitz moduli of the set-valued mappings $\mathcal{E}_{\mathcal{F}}$
and $\mathcal{E}_{\mathcal{P}}$ under consideration. Both Sections~\ref%
{pareto} and \ref{lp} contain illustrative examples of their own interest.
The final Section~\ref{conc} summarizes the obtained results and discusses
some directions of future research.\vspace*{0.05in}

Throughout the paper we use the standard notion in variational analysis and
optimization. Recall that $\mathrm{conv\,}\Omega $, $\mathrm{cone\,}\Omega $%
, and $\mathrm{span\,}\Omega $ stand, respectively, for the \emph{convex hull%
}, the \emph{conic convex hull}, and the \emph{linear subspace} generated by
the set $\Omega \subset \mathbb{R}^{n}$ under the convention that $\mathrm{%
cone\,}\emptyset =\{0_{n}\}$, where $0_{n}$ is the origin of $\mathbb{R}^{n}$%
. If $\Omega $ is convex, by $O^{+}(\Omega )$ we represent the \emph{%
recession cone }of $\Omega $. The space of decision variables $\mathbb{R}%
^{n} $ is endowed with an arbitrary norm $\left\Vert \cdot \right\Vert $,
while the space of parameters $\mathbb{R}^{T}$ is equipped with the supremum
norm 
\begin{equation}
\left\Vert b\right\Vert _{\infty }:=\sup_{t\in T}\left\vert b_{t}\right\vert
.  \label{sup-norm}
\end{equation}

\section{Preliminaries and First Results}

\label{prelim}

In this section, unless otherwise stated, $\mathcal{M}\colon
Y\rightrightarrows Z$ is a set-valued mapping between Banach spaces $Y$ and $%
Z$ which topological duals are denoted by $Y^{\ast }$ and $Z^{\ast }$,
respectively. The \emph{coderivative} of $\mathcal{M}$ at $\left( \overline{y%
},\overline{z}\right) \in \mathrm{gph}\mathcal{M}$ is a positively
homogeneous multifunction $D^{\ast }\mathcal{M}\left( \overline{y},\overline{%
z}\right) \colon Z^{\ast }\rightrightarrows Y^{\ast }$ defined by 
\begin{equation}
y^{\ast }\in D^{\ast }\mathcal{M}\left( \overline{y},\overline{z}\right)
\left( z^{\ast }\right) \Longleftrightarrow \left( y^{\ast },-z^{\ast
}\right) \in N\big(\left( \overline{y},\overline{z}\right) ;\mathrm{gph}%
\mathcal{M}\big),  \label{cod}
\end{equation}%
where $N\left( \left( \overline{y},\overline{z}\right) ;\mathrm{gph}\mathcal{%
M}\right) $ is the (basic, limiting, Mordukhovich) \emph{normal cone} to $%
\mathrm{gph}\mathcal{M}$ at $\left( \overline{y},\overline{z}\right) $; see,
e.g., \cite{Mor06i} and \cite{rw}. For simplicity, $\left\Vert \cdot
\right\Vert $ stands for the norm in any Banach space $X$, and $\left\Vert
\cdot \right\Vert _{\ast }$ is the corresponding dual norm, i.e., 
\begin{equation*}
\left\Vert x^{\ast }\right\Vert _{\ast }=\sup \big\{\left\langle x^{\ast
},x\right\rangle \big|\;\left\Vert x\right\Vert \leq 1,\text{ }x\in X\big\},%
\text{ }x^{\ast }\in X^{\ast },
\end{equation*}%
where $\left\langle .,.\right\rangle $ denotes the canonical pairing between 
$X$ and $X^{\ast }$. If no confusion arises, from now on we skip the
subscript `$_{\ast }$' in the dual norm notation.

When both spaces $Y$ and $Z$ are finite-dimensional and the graph of $%
\mathcal{M}$ is locally closed around $(\overline{y},\overline{z})\in 
\mathrm{gph}\mathcal{M}$, there is the following \emph{precise formula} for
the \emph{computing the Lipschitz modulus} of $\mathcal{M}\left( \overline{y}%
,\overline{z}\right) $: 
\begin{equation}
\mathrm{lip}\mathcal{M}\left( \overline{y},\overline{z}\right) =\left\Vert
D^{\ast }\mathcal{M}\left( \overline{y},\overline{z}\right) \right\Vert
:=\sup \big\{\left\Vert y^{\ast }\right\Vert _{\ast }\big|y^{\ast }\in
D^{\ast }\mathcal{M}\left( \overline{y},\overline{z}\right) (z^{\ast
}),\Vert z^{\ast }\Vert _{\ast }=1\big\},  \label{eq_lipM}
\end{equation}%
which was obtained in \cite{Mo93}. We also refer the reader to \cite[%
Theorem~9.40]{rw} for another proof of this result, which was labeled
therein as the Mordukhovich criterion. An infinite-dimensional extension of %
\eqref{eq_lipM} was derived in \cite[Theorem~4.10]{Mor06i}. It is more
involved and is not used in this paper dealing with finite-dimensional
multiobjective optimization problems of type \eqref{eq_MOP}. A simplified
proof of \eqref{eq_lipM} in finite dimensions was given in \cite[Theorem~3.3]%
{Mor18}.

If the graph of $\mathcal{M}$ is \emph{convex}, the normal cone in %
\eqref{cod} reduces to the normal convex of convex analysis, and hence $%
y^{\ast}\in D^{\ast}\mathcal{M}\left(\overline{y},\overline{z}%
\right)\left(z^{\ast}\right)$ if and only if 
\begin{equation*}
\left\langle\left(y^{\ast},-z^{\ast}\right),\left(y^{\prime}-\overline{y}%
,z^{\prime}-\overline{z}\right)\right\rangle\le 0\;\text{ for all }%
\;\left(y^{\prime},z^{\prime}\right)\in\mathrm{gph}\mathcal{M},
\end{equation*}
which is equivalent to the description 
\begin{equation}  \label{eq_conderiv_convex}
\left\langle y^{\ast},y^{\prime}-\overline{y}\right\rangle\le\left\langle
z^{\ast},z^{\prime}-\overline{z}\right\rangle\;\text{ for all }%
\;\left(y^{\prime},z^{\prime }\right)\in\mathrm{gph}\mathcal{M}.
\end{equation}

Given further a closed and convex \emph{ordering cone} $\Theta\subset Z$,
the \emph{epigraphical multifunction}\emph{\ }$\mathcal{E}_{\mathcal{M}%
}\colon X\rightrightarrows Z$ associated with $\mathcal{M}$ and the cone $%
\Theta$ is that which graph $\mathrm{gph}\mathcal{E}_{\mathcal{M}}$
coincides with the \emph{epigraph} of $\mathcal{M}$ with respect to $\Theta$%
. In other words, we have $\mathrm{\mathcal{E}_{\mathcal{M}}}\left(y\right):=%
\mathcal{M}\left(y\right)+\Theta$ and 
\begin{equation*}
\mathrm{epi\mathcal{M}}:=\mathrm{gph\mathcal{E}_{\mathcal{M}}}=\big\{%
\left(y,z\right)\big|\;z\in\mathcal{M}\left(y\right)+\Theta\big\},
\end{equation*}
where we skip indicating $\Theta$ in the epigraphical notation.\vspace*{%
0.05in}

In accordance with \cite{BaMo07}, we present the following definition of the
subdifferential of $\mathcal{M}$ at the reference point of its epigraph with
respect to $\Theta$.

\begin{definiti}
\label{def fre subdiff} Let $\left( \overline{y},\overline{z}\right) \in 
\mathrm{epi\mathcal{M}}$ be given. The \textsc{subdifferential} of $\mathcal{%
M}$ at $\left( \overline{y},\overline{z}\right) $ denoted as $\partial 
\mathcal{M}\left( \overline{y},\overline{z}\right) $ is a subset of $Y^{\ast
}$ defined by 
\begin{equation}
\partial \mathcal{M}\left( \overline{y},\overline{z}\right) :=\big\{y^{\ast
}\in D^{\ast }\mathcal{E}_{\mathcal{M}}\left( \overline{y},\overline{z}%
\right) \left( z^{\ast }\right) \big|-z^{\ast }\in N\left( 0;\Theta \right)
,\;\Vert z^{\ast }\Vert _{\ast }=1\big\},  \label{eq_def_subdif_gen}
\end{equation}%
where $N\left( 0;\Theta \right) \subset Z^{\ast }$ is the convex normal cone
to the set $\Theta $ at the origin of $Z$.
\end{definiti}

Note that if $\mathcal{M}\colon Y\rightarrow\mathbb{R}$ is a proper convex
function with\textbf{\ }$\Theta=\mathbb{R}_{+}$, then $\mathrm{gph\mathcal{E}%
_{\mathcal{M}}}$ is its standard epigraph, and for any $\overline{y}\in%
\limfunc{dom}\mathcal{M}$ the set $\partial \mathcal{M}\left(\overline{y},%
\mathcal{M}(\overline{z})\right)$ is the classical subdifferential of\textbf{%
\ }$\mathcal{M}$ at $\overline{y}$ in the sense of convex analysis.

Observe also that the set $-N\left(0;\Theta\right) $ is nothing else but the
polar cone $\Theta^{\circ }$, and thus we have the following representation
of the coderivative of $\mathrm{\mathcal{E}_{\mathcal{M}}}$ in terms of the
graph $\mathrm{gph}\mathcal{M}$ instead of the epigraph $\mathrm{epi}%
\mathcal{M}$.

\begin{prop}
\label{coder-epi} Assume that $\mathrm{epi\mathcal{M}}$ is a convex set, and
let $\left( \overline{y},\overline{z}\right)\in\mathrm{epi\mathcal{M}}$.
Then for any $z^{\ast}\in Z^{\ast}$ we have the representation 
\begin{align*}
&D^{\ast}\mathcal{E}_{\mathcal{M}}\left(\overline{y},\overline{z}\right)
\left(z^{\ast}\right)\medskip \\
&=\left\{%
\begin{array}{l}
\{y^{\ast}\in Y^{\ast}\mid\left\langle y^{\ast},y^{\prime}-\overline{y}%
\right\rangle\le\left\langle z^{\ast},z^{\prime}-\overline{z} \right\rangle%
\text{ }\forall\left(y^{\prime},z^{\prime}\right)\in\mathrm{gph}\mathcal{M\}}%
\;\text{ if }\;z^{\ast}\in\Theta^{\circ},\medskip \\ 
\emptyset\;\text{ if }\;z^{\ast}\notin\Theta^{\circ}.%
\end{array}
\right.
\end{align*}
\end{prop}

\begin{proof}
Take $z^{\ast }\in \Theta ^{\circ }$. By using the definitions of the
coderivative \eqref{cod} and of the epigraphical multifunction $\mathcal{E}_{%
\mathcal{M}}$, we get due to the convexity of $\mathrm{epi\mathcal{M}}$ ($=%
\mathrm{gph}\mathcal{E}_{\mathcal{M}}$) that 
\begin{equation}
D^{\ast }\mathcal{E}_{\mathcal{M}}\left( \overline{y},\overline{z}\right)
\left( z^{\ast }\right) =\big\{y^{\ast }\mid \left\langle y^{\ast
},y^{\prime }-\overline{y}\right\rangle \leq \left\langle z^{\ast
},z^{\prime }-\overline{z}\right\rangle \;\forall \left( y^{\prime
},z^{\prime }\right) \in \mathrm{gph}\mathcal{E}_{\mathcal{M}}\big\}.
\label{eq_EM_M}
\end{equation}%
Let us show that $\mathrm{gph}\mathcal{E}_{\mathcal{M}}$ can be equivalently
replaced by $\mathrm{gph}\mathcal{M}$ in \eqref{eq_EM_M}. Indeed, take any $%
y^{\ast }\in Y^{\ast }$ satisfying \eqref{eq_conderiv_convex}. Pick further
any $\left( \widetilde{y},\widetilde{z}\right) \in \mathrm{gph}\mathcal{E}_{%
\mathcal{M}}$ and write $\widetilde{z}=z^{\prime }+u\,\ $with $z^{\prime
}\in \mathcal{M}\left( \widetilde{y}\right) $ and $u\in \Theta $. Then we
obtain the inequalities 
\begin{equation*}
\left\langle y^{\ast },\widetilde{y}-\overline{y}\right\rangle \leq
\left\langle z^{\ast },z^{\prime }-\overline{z}\right\rangle \leq
\left\langle z^{\ast },\widetilde{z}-\overline{z}\right\rangle
\end{equation*}%
due to $\left\langle z^{\ast },u\right\rangle \geq 0$. It gives us the
claimed coderivative formula for $z^{\ast }\in \Theta ^{\circ }$.

Suppose now that $z^{\ast }\notin \Theta ^{\circ }$ and find $u\in \Theta $
such that $\left\langle z^{\ast },u\right\rangle <0$. Arguing by
contradiction, assume that there is $y^{\ast }\in D^{\ast }\mathcal{E}_{%
\mathcal{M}}\left( \overline{y},\overline{z}\right) \left( z^{\ast }\right) $%
, i.e., by \eqref{eq_conderiv_convex} we have 
\begin{equation*}
\left\langle y^{\ast },y^{\prime }-\overline{y}\right\rangle \leq
\left\langle z^{\ast },z^{\prime }-\overline{z}\right\rangle \;\text{ for
all }\;\left( y^{\prime },z^{\prime }\right) \in \mathrm{gph}\mathcal{E}_{%
\mathcal{M}}.
\end{equation*}%
Since $\mathcal{E}_{\mathcal{M}}\left( \overline{y}\right) +\Theta =\mathcal{%
E}_{\mathcal{M}}\left( \overline{y}\right) $, it follows that $\left( 
\overline{y},\overline{z}+u\right) \in \mathrm{gph}\mathcal{E}_{\mathcal{M}}$%
, and therefore 
\begin{equation*}
0=\left\langle y^{\ast },\overline{y}-\overline{y}\right\rangle \leq
\left\langle z^{\ast },u\right\rangle ,
\end{equation*}%
which is a contradiction that completes the proof of the proposition.
\end{proof}

\vspace*{0.05in}

Employing Proposition~\ref{coder-epi} leads us to deriving effective
representations of the subdifferential of $\mathcal{M}$ and the Lipschitz
modulus of $\mathcal{E}_{\mathcal{M}}$ as well as to a relation between the
latter and the Lipschitz modulus of $\mathcal{M}$ at the reference point.

\begin{theo}
\label{Cor_lip_subdiff} Let the epigraphical set $\mathrm{epi\mathcal{M}}$
be convex, and let $\left( \overline{y},\overline{z}\right) \in \mathrm{epi%
\mathcal{M}}$. Then we have the subdifferential representation 
\begin{equation}
\partial \mathcal{M}\left( \overline{y},\overline{z}\right) =\bigcup_{\QATOP{%
\scriptstyle{z^{\ast }\in \Theta ^{\circ }}}{\scriptstyle{\Vert z^{\ast
}\Vert }_{\ast }{=1}}}\big\{y^{\ast }\big|\;\left\langle y^{\ast },y^{\prime
}-\overline{y}\right\rangle \leq \left\langle z^{\ast },z^{\prime }-%
\overline{z}\right\rangle \forall \left( y^{\prime },z^{\prime }\right) \in 
\mathrm{gph}\mathcal{M}\big\}.  \label{eq_subdif_convex}
\end{equation}%
If in addition $Y$ and $Z$ are finite-dimensional and if the set $\mathrm{epi%
\mathcal{M}}$ is locally closed around $\left( \overline{y},\overline{z}%
\right) $, then the Lipschitz modulus of $\mathcal{E}_{\mathcal{M}}$ at $%
\left( \overline{y},\overline{z}\right) $ is computed by 
\begin{equation}
\mathrm{lip}\mathcal{E}_{\mathcal{M}}\left( \overline{y},\overline{z}\right)
=\sup \big\{\Vert y^{\ast }\Vert _{\ast }\;\big|\;y^{\ast }\in \partial 
\mathcal{M}\left( \overline{y},\overline{z}\right) \big\}.  \label{eq_lip_EM}
\end{equation}%
Assuming furthermore that $\left( \overline{y},\overline{z}\right) \in 
\mathrm{gph}\mathcal{M}$ and that the set $\mathrm{gph}\mathcal{M}$ is
locally closed around this point, we conclude that 
\begin{equation}
\mathrm{lip}\mathcal{E}_{\mathcal{M}}\left( \overline{y},\overline{z}\right)
\leq \mathrm{lip}\mathcal{M}\left( \overline{y},\overline{z}\right) .
\label{eq_lipEmenorlip}
\end{equation}
\end{theo}

\begin{proof}
Representation \eqref{eq_subdif_convex} follows directly from definition %
\eqref{eq_def_subdif_gen} of the subdifferential $\partial\mathcal{M}\left(%
\overline{y},\overline{z}\right)$ combined with Proposition~\ref{coder-epi}.

Assuming now that the spaces $Y$ and $Z$ are finite-dimensional, applying
the Lipschitz modulus formula \eqref{eq_lipM} to the epigraphical mapping $%
\mathcal{E}_{\mathcal{M}}$, and appealing again to Proposition~\ref%
{coder-epi} tell us that 
\begin{eqnarray*}
\mathrm{lip}\mathcal{E}_{\mathcal{M}}\left( \overline{y},\overline{z}\right)
&=&\sup \big\{\Vert y^{\ast }\Vert _{\ast }\;\big|\;y^{\ast }\in D^{\ast }%
\mathcal{E}_{\mathcal{M}}\left( \overline{y},\overline{z}\right) (z^{\ast
}),\;\Vert z^{\ast }\Vert _{\ast }=1\big\} \\
&=&\sup \big\{\Vert y^{\ast }\Vert _{\ast }\;\big|\;y^{\ast }\in D^{\ast }%
\mathcal{E}_{\mathcal{M}}\left( \overline{y},\overline{z}\right) (z^{\ast
}),\;z^{\ast }\in \Theta ^{\circ },\;\Vert z^{\ast }\Vert _{\ast }=1\big\}.
\end{eqnarray*}%
Thus the claimed formula \eqref{eq_lip_EM} follows from the definition of $%
\partial \mathcal{M}\left( \overline{y},\overline{z}\right) $.

To verify finally the inequality \eqref{eq_lipEmenorlip}, denote by $%
\widehat{N}\left(\left(\overline{y},\overline{z}\right);\Omega\right)$ the
prenormal/regular normal cone to $\Omega\subset Y\times Z$ at $\left(%
\overline{y},\overline{z}\right)$ (see, e.g., \cite{Mor06i,rw}) and using
the convexity of $\mathrm{epi\mathcal{M}}$, we get 
\begin{equation*}
N\left(\left(\overline{y},\overline{z}\right);\mathrm{epi}\mathcal{M}\right)=%
\widehat{N}\left(\left(\overline{y},\overline{z}\right);\mathrm{epi}\mathcal{%
M}\right) \subset\widehat{N}\left(\left(\overline{y}, \overline{z}\right);%
\mathrm{gph}\mathcal{M}\right)=N\left(\left(\overline{y},\overline{z}\right);%
\mathrm{gph}\mathcal{M}\right),
\end{equation*}
where the inclusion comes from \cite[Proposition~1.5]{Mor06i}. This gives us 
\begin{equation*}
D^{\ast}\mathcal{E}_{\mathcal{M}}\left(\overline{y},\overline{z}%
\right)\left(z\right)\subset D^{\ast}\mathcal{M}\left(\overline{y},\overline{%
z}\right)\left(z\right)\;\text{ for all }\;z\in Z
\end{equation*}
and thus deduces \eqref{eq_lipEmenorlip} from the basic coderivative formula %
\eqref{eq_lipM}.
\end{proof}

\begin{rem}
\label{Rem_lip_EM_M}\textrm{The inequality in \eqref{eq_lipEmenorlip} may be 
\emph{strict} as illustrated by the following simple example. Consider $%
\Theta:=\mathbb{R}_{+}$ and $\mathcal{M}\colon\mathbb{R}\rightrightarrows%
\mathbb{R}$ given by $\mathcal{M}(y):=[y,2y]$ if $y\ge 0$ and $\mathcal{M}%
(y):=\{0\}$ if $y<0$. Then it is easy to calculate that 
\begin{equation*}
1=\mathrm{lip}\mathcal{E}_{\mathcal{M}}\left( 0,0\right)<\mathrm{lip}%
\mathcal{M}\left(0,0\right)=2.
\end{equation*}%
}
\end{rem}

\section{Stability Analysis of Epigraphical Feasible Sets}

\label{feas}

The underlying goal of this section is explicit computing the Lipschitz
modulus of epigraphical feasible set mapping associated with the
parameterized MLP problem \eqref{eq_MOP}. As we know from Sections~\ref%
{intro} and \ref{prelim}, our approach reduces this computation to deriving
a verifiable formula to calculate the subdifferential in the sense of
Definition~\ref{def fre subdiff} of the perturbed feasible set $\mathcal{F}$
in terms of its given data. Proceeding in this way, we concentrate here on
obtaining the representations of the subdifferential and Lipschitz modulus
with involving the graph of the nondominated solution mapping $\mathcal{S}$.

Let us begin with two lemmas. The first one is a well-known result that
gives a characterization of nondominated solutions to MLP$\left(b\right)$
via optimal solutions to a scalarized linear program. We formulate it
without a proof. The second lemma is a new result, which plays a key role
throughout the paper.

\begin{lem}
\label{Lem_optimality_conditions} Let $x_{0}\in\mathcal{F}\left(b\right)$
for some $b\in\mathbb{R}^{T}$. Then the following are equivalent:

\textbf{(i)} $x_{0}\in\mathcal{S}\left(b\right)$.

\textbf{(ii)} There exist numbers $\lambda _{i}>0$ for $i=1,...,q,$ such
that 
\begin{equation*}
x_{0}\in\arg\min \left\{\dsum_{i=1}^{q}\lambda _{i}\left\langle
c_{i},x\right\rangle\Big|\;x\in\mathcal{F}\left(b\right)\right\}.
\end{equation*}
\end{lem}

\vspace*{0.05in}

To formulate the second lemma, recall that 
\begin{equation*}
\mathrm{dom}\mathcal{S}:=\big\{b\in \mathbb{R}^{T}\big|\;\mathcal{S}%
\left(b\right)\ne\emptyset\big\}.
\end{equation*}

\begin{lem}
\label{Lem_F_notS} Let $b\in\mathrm{dom}\mathcal{S}$. Then for any $x_{0}\in%
\mathcal{F}\left( b\right)\setminus \mathcal{S}\left(b\right)$ there exists $%
\widetilde{x}_{0}\in\mathcal{S}\left(b\right)$ such that $\left\langle c_{i},%
\widetilde{x}_{0}\right\rangle\le\left\langle c_{i},x_{0}\right\rangle$
whenever $i=1,...,q$.
\end{lem}

\begin{proof}
Fix $x_{0}\in\mathcal{F}\left(b\right)\setminus\mathcal{S}\left(b\right)$
and proceed step-by-step as follows:\vspace*{0.05in}

\emph{Step}~1. Let us proof the existence of solutions to the linear
program: 
\begin{equation}
x_{1}\in \arg \min \big\{\left\langle c_{1},x\right\rangle \big|\;x\in 
\mathcal{F}\left( b\right) ,\text{ }\left\langle c_{i},x\right\rangle \leq
\left\langle c_{i},x_{0}\right\rangle ,\text{ }i=1,...,q\big\}.
\label{eq_001}
\end{equation}%
Arguing by contradiction, suppose that \eqref{eq_001} has no optimal
solutions. Since $x_{0}$ is a feasible solution to \eqref{eq_001}, our
assumption is equivalent to the unboundedness of the set of feasible
solutions to the linear program \eqref{eq_001}. Thus there exists a sequence 
$\{w_{r}\}_{r\in \mathbb{N}}\subset \mathbb{R}^{n}$ such that 
\begin{equation*}
w_{r}\in \mathcal{F}\left( b\right) ,\text{ }\left\langle
c_{i},w_{r}\right\rangle \leq \left\langle c_{i},x_{0}\right\rangle ,\text{ }%
i=1,...,q,\text{ for all }r\in \mathbb{N},
\end{equation*}%
while we have the infinite limit 
\begin{equation*}
\lim_{r\rightarrow \infty }\left\langle c_{1},w_{r}\right\rangle =-\infty .
\end{equation*}%
Remembering that $\mathcal{S}\left( b\right) \neq \emptyset $, pick any $%
\widetilde{x}\in \mathcal{S}\left( b\right) $ and find by Lemma~\ref%
{Lem_optimality_conditions} numbers $\lambda _{i}>0$ with $i=1,...,q$ such
that 
\begin{equation*}
\widetilde{x}\in \arg \min \left\{ \dsum_{i=1}^{q}\lambda _{i}\left\langle
c_{i},x\right\rangle \Big|\;x\in \mathcal{F}\left( b\right) \right\} .
\end{equation*}%
This readily brings us to the contradiction: 
\begin{equation*}
\dsum_{i=1}^{q}\lambda _{i}\left\langle c_{i},\widetilde{x}\right\rangle
\leq \dsum_{i=1}^{q}\lambda _{i}\left\langle c_{i},w_{r}\right\rangle \leq
\lambda _{1}\left\langle c_{1},w_{r}\right\rangle +\dsum_{i=2}^{q}\lambda
_{i}\left\langle c_{i},x_{0}\right\rangle \underset{r\rightarrow \infty }{%
\longrightarrow }-\infty ,
\end{equation*}%
which therefore verifies the existence of the solution $x_{1}$ to %
\eqref{eq_001}. Note furthermore that if $x_{1}$ satisfies $x_{1}\in 
\mathcal{S}\left( b\right) $, then the proof of the lemma is complete.
Otherwise we go to the next step as follows.\vspace*{0.05in}

\emph{Step}~2. Suppose that $x_{1}\in\mathcal{F}\left(b\right)\setminus%
\mathcal{S}\left(b\right)$. Then arguing as in Step~1 ensures the existence
of a vector $x_{2}\in{\mathbb{R}}^n$ satisfying 
\begin{equation}  \label{eq_002}
x_{2}\in \arg\min\big\{\left\langle c_{2},x\right\rangle\big|\;x\in\mathcal{F%
}\left(b\right),\text{ }\left\langle c_{i},x\right\rangle\le\left\langle
c_{i},x_{1}\right\rangle,\text{ }i=1,...,q\big\}.
\end{equation}
As before, the proof of the lemma is finished if $x_{2}\in\mathcal{S}%
\left(b\right)$. Otherwise we go to \emph{Step}~3 and proceed similarly.

Reaching in this way \emph{Step}~$j$ with some $j<q$, we either finish the
proof, or arrive at \emph{Step}~$q$ that is described below.\vspace*{0.05in}

\emph{Step}~$q$. Suppose that $x_{q-1}\in\mathcal{F}\left(b\right)\setminus%
\mathcal{S}\left(b\right)$. Again we get 
\begin{equation*}
x_{q}\in\arg\min\big\{\left\langle c_{q},x\right\rangle\big|\;x\in\mathcal{F}%
\left(b\right),\text{ }\left\langle c_{i},x\right\rangle\le\left\langle
c_{i},x_{q-1}\right\rangle,\text{ }i=1,...,q\big\}.
\end{equation*}

Let us show that now we do not have any choice but $x_{q}\in\mathcal{S}%
\left(b\right)$. Arguing by contradiction, assume that there exists $w\in%
\mathcal{F}\left(b\right)$ such that 
\begin{equation*}
\left\{%
\begin{array}{l}
\left\langle c_{i},w\right\rangle\le\left\langle c_{i},x_{q}\right\rangle%
\text{ for all }\;i=1,...,q, \\ 
\left\langle c_{j},w\right\rangle<\left\langle c_{j},x_{q}\right\rangle\text{
for some }\;j\in\{1,...,q\}.%
\end{array}%
\right.
\end{equation*}
Then we arrive at a contradiction with the choice of $x_{j}$. Indeed, it
follows that 
\begin{eqnarray*}
\left\langle c_{i},w\right\rangle &\le&\left\langle
c_{i},x_{q}\right\rangle\le\left\langle c_{i},x_{q-1}\right\rangle\le
...\le\left\langle c_{i},x_{j-1}\right\rangle\text{ for all }i=1,...,q, \\
\left\langle c_{j},w\right\rangle&<&\left\langle
c_{j},x_{q}\right\rangle\le\left\langle
c_{j},x_{q-1}\right\rangle\le...\le\left\langle c_{j},x_{j}\right\rangle.
\end{eqnarray*}
This completes the proof of the lemma.\vspace*{0.05in}
\end{proof}

The next theorem provides a description of the subdifferential $\partial%
\mathcal{F}\left(\overline{b},\overline{x}\right)$ in terms of $\mathrm{gph}%
\mathcal{S}$ (instead of $\mathrm{gph}\mathcal{F}$ as in the definition),
which eventually allows us to relate the subdifferential $\partial\mathcal{F}%
\left(\overline{b},\overline{x}\right)$ to the subdifferential of the Pareto
front mapping \eqref{eq_pareto}. This leads us to new results even in the
case of standard linear programs as shown in Section~\ref{lp}.

\begin{rem}
\label{Rem_C} \emph{Using the notation of Section~\ref{prelim} gives us} 
\begin{equation*}
\mathrm{\mathcal{E}_{\mathcal{F}}}\left(b\right)=\mathcal{F}%
\left(b\right)+\Theta\;\;\mbox{\rm for all }\;b\in\mathbb{R}^{T}\;\;%
\mbox{\rm with }\;\Theta:=\big\{c_{1},...,c_{q}\big\}^\circ.
\end{equation*}
\emph{From now on we denote} 
\begin{equation*}
C:=-N\left(0_{n};\Theta\right)=\Theta^{\circ}=\mathrm{cone}\big\{%
c_{1},...c_{q}\big\},
\end{equation*}
\emph{where the last equality immediately follows from the classical Farkas
Lemma.}
\end{rem}

Here is the aforementioned theorem with the subdifferential calculation. In
the paper, and despite $\mathbb{R}^{n}$ is self-dual, we are using $%
\left\Vert c\right\Vert _{\ast }$ and $\left\Vert a_{t}\right\Vert _{\ast }$
because $c$ and $a_{t}$ are regarded as linear functions ($x\mapsto
\left\langle c,x\right\rangle $ and $x\mapsto \left\langle
a_{t},x\right\rangle $, respectively).\textbf{\ }

\begin{theo}
\label{Theo_subdiffF} Let $\left( \overline{b},\overline{x}\right) \in 
\mathrm{gph\mathcal{E}_{\mathcal{F}}}$. Then we have the subdifferential
formula 
\begin{equation}
\partial \mathcal{F}\left( \overline{b},\overline{x}\right) =\dbigcup_{%
\QATOP{\scriptstyle{c\in C}}{\scriptstyle{\Vert c\Vert }_{\ast }{=1}}}\big\{%
y\in \mathbb{R}^{T}\big|\left\langle y,b-\overline{b}\right\rangle \leq
\left\langle c,x-\overline{x}\right\rangle \;\text{for all}\;\left(
b,x\right) \in \mathrm{gph}\mathcal{S}\big\}.  \label{eq_subdiffF}
\end{equation}
\end{theo}

\begin{proof}
By the convexity of the sets $\mathrm{gph}\mathcal{F}$ and $\mathrm{gph}%
\mathcal{E}_{\mathcal{F}}$ we get from \eqref{eq_subdif_convex} that 
\begin{equation*}
\partial \mathcal{F}\left( \overline{b},\overline{x}\right) =\dbigcup_{%
\QATOP{\scriptstyle{c\in C}}{\scriptstyle{\Vert c\Vert }_{\ast }{=1}}}\big\{%
y\in \mathbb{R}^{T}\big|\;\left\langle y,b-\overline{b}\right\rangle \leq
\left\langle c,x-\overline{x}\right\rangle \;\text{for all }\;\left(
b,x\right) \in \mathrm{gph}\mathcal{F}\big\}.
\end{equation*}%
Since $\mathrm{gph}\mathcal{S\subset }\mathrm{gph}\mathcal{F}$, we only need
to verify the inclusion `$\supset $' of \eqref{eq_subdiffF}.

To proceed, pick any $c\in C$ with ${\Vert c\Vert }_{\ast }=1$ and select $%
y\in \mathbb{R}^{T}$ such that 
\begin{equation}
\left\langle y,b-\overline{b}\right\rangle \leq \left\langle c,x-\overline{x}%
\right\rangle \;\text{for all }\left( b,x\right) \in \mathrm{gph}\mathcal{S}.
\label{eq_inequalities}
\end{equation}%
Arguing by contradiction, suppose that there exists $\left(
b_{0},x_{0}\right) \in \mathrm{gph}\mathcal{F}$ with 
\begin{equation*}
\left\langle y,b_{0}-\overline{b}\right\rangle >\left\langle c,x_{0}-%
\overline{x}\right\rangle ,
\end{equation*}%
which yields $\left( b_{0},x_{0}\right) \notin \mathrm{gph}\mathcal{S}$.
Applying then Lemma~\ref{Lem_F_notS} to $\left( b_{0},x_{0}\right) $ ensures
the existence of $\widetilde{x}_{0}\in \mathcal{S}\left( b_{0}\right) $ such
that $\left\langle c_{i},\widetilde{x}_{0}\right\rangle \leq \left\langle
c_{i},x_{0}\right\rangle $ for all $i=1,...,q$. In particular, we get $%
\left\langle c,\widetilde{x}_{0}\right\rangle \leq \left\langle
c,x_{0}\right\rangle $. Therefore 
\begin{equation*}
\left\langle y,b_{0}-\overline{b}\right\rangle >\left\langle c,x_{0}-%
\overline{x}\right\rangle \geq \left\langle c,\widetilde{x}_{0}-\overline{x}%
\right\rangle ,
\end{equation*}%
which contradicts \eqref{eq_inequalities} and thus completes the proof of
the theorem.
\end{proof}

\vspace*{0.05in}

Now we are ready to establish a precise formula for computing the Lipschitz
modulus at $\left(\overline{b},\overline{x}\right)\in\mathrm{gph\mathcal{E}_{%
\mathcal{F}}}$ of the epigraphical feasible set mapping from 
\eqref{eq_epi
mapp}. In the next theorem we employ the $l_{1}$-norm $\left\Vert\cdot\right%
\Vert_{1}$ on $\mathbb{R}^{T}$, which is dual to the primal supremum norm %
\eqref{sup-norm} used above.

\begin{theo}
\label{Cor_lip_EF} Let $\left( \overline{b},\overline{x}\right) \in \mathrm{%
gph\mathcal{E}_{\mathcal{F}}}$. Then we have 
\begin{eqnarray*}
\mathrm{lip}\mathcal{E}_{\mathcal{F}}\left( \overline{b},\overline{x}\right)
&=&\sup \big\{\left\Vert y\right\Vert _{1}\big|\;y\in \partial \mathcal{F}%
\left( \overline{b},\overline{x}\right) \big\} \\
&=&\sup \big\{\dbigcup_{\QATOP{\scriptstyle{c\in C}}{\scriptstyle{\Vert
c\Vert }_{\ast }{=1}}}\big\{\left\Vert y\right\Vert _{1}\big|\;\left\langle
y,b-\overline{b}\right\rangle \leq \left\langle c,x-\overline{x}%
\right\rangle \text{ }\forall \left( b,x\right) \in \mathrm{gph}\mathcal{S}%
\big\},
\end{eqnarray*}%
and thus the multifunction $\mathcal{E}_{\mathcal{F}}$ is Lipschitz-like
around $\left( \overline{b},\overline{x}\right) $ if and only if 
\begin{equation*}
\sup \big\{\dbigcup_{\QATOP{\scriptstyle{c\in C}}{\scriptstyle{\Vert c\Vert }%
_{\ast }{=1}}}\Big\{\left\Vert y\right\Vert _{1}\Big|\;\left\langle y,b-%
\overline{b}\right\rangle \leq \left\langle c,x-\overline{x}\right\rangle 
\text{ }\forall \;\left( b,x\right) \in \mathrm{gph}\mathcal{S}\Big\}<\infty
.
\end{equation*}
\end{theo}

\begin{proof}
Observe that\ $\left( b,x\right) \in \mathrm{gph}\mathcal{F}%
\Longleftrightarrow a_{t}^{\prime }x-e_{t}^{\prime }b\leq 0$ for all $%
t=1,...,m$, where $e_{t}\in \mathbb{R}^{m}$\textbf{\ }is the $t$-th vector
of the canonical basis of $\mathbb{R}^{m}$. Appealing to Remark~\ref{Rem_C},
we see that the set $\mathrm{gph}\mathcal{E}_{\mathcal{F}}$ is a polyhedral
convex cone admitting the representation 
\begin{equation*}
\mathrm{gph}\mathcal{E}_{\mathcal{F}}=\mathrm{gph}\mathcal{F+}\left(
\{0_{m}\}\times \big\{c_{1},...,c_{q}\big\}^{\circ }\right) ;
\end{equation*}%
so this set is closed and convex. Thus the claimed modulus formula follows
from \eqref{eq_lip_EM} and Theorem~\ref{Theo_subdiffF}. The last statement
of this theorem follows directly from the definition of the Lipschitz
modulus and the formula for its computation.
\end{proof}

\vspace*{0.05in}

Examples~\ref{Exa1} and \ref{Exa2} illustrate both Theorem~\ref%
{Theo_subdiffF} and Theorem~\ref{Cor_lip_EF}. They are included in the next
section for comparative purposes, specifically to point out the similarities
between the subdifferentials $\partial\mathcal{F}$ and $\partial\mathcal{P}$.

\section{Computation Formulas for Feasible Sets}

\label{scal}

In this section we derive a precise formula for representing the
epigraphical multifunction $\mathcal{E}_{\mathcal{F}}$ from 
\eqref{eq_epi
mapp} via solutions of a new linear inequality system associated with $%
\mathcal{F}(b)$. More constructive representations are obtained for some
specific forms of feasible solution sets that are especially important for
applications. All of this constitutes, in particular, the basis for
computations of the optimal value in linear programs, which is illustrated
and further developed in Section~\ref{lp} in the framework of Example~\ref%
{Exa_new}.

Let us start revealing the following relationship between the
`multiobjective epigraphical feasible set mapping' $\mathcal{E}_{\mathcal{F}%
} $ and its linear program counterpart $\mathcal{F}\left(b\right)+\left\{c%
\right\}^{\circ }$ coming from a certain scalarization technique.

\begin{theo}
\label{Th_MOP value} For any $b\in\mathbb{R}^{T}$ we have the relationship 
\begin{equation*}
\mathcal{E}_{\mathcal{F}}\left(b\right)=\bigcap_{c\in\mathrm{conv}%
\left\{c_{1},...,c_{q}\right\}}\left(\mathcal{F}\left(b\right)+\left\{c%
\right\}^{\circ}\right).
\end{equation*}
\end{theo}

\begin{proof}
Confining ourselves to the nontrivial case where $\mathcal{F}%
\left(b\right)\ne\emptyset$, observe first that the inclusion `$\subset$'
follows from the obvious fact that 
\begin{equation*}
\left\{c_{1},...,c_{q}\right\}^{\circ}=\bigcap_{c\in\mathrm{conv}%
\left\{c_{1},...,c_{q}\right\}}\left\{c\right\}^{\circ}.
\end{equation*}
To verify the opposite inclusion `$\supset$', assume that $x\notin\mathcal{E}%
_{\mathcal{F}}\left(b\right)$ and then show that there exists $c\in\mathrm{%
conv}\left\{c_{1},...,c_{q}\right\}$ such that $x\notin\mathcal{F}%
\left(b\right)+\left\{c\right\}^{\circ }$. Denote by $\widehat{x}$ the
Euclidean projection of $x$ onto $\mathcal{E}_{\mathcal{F}}\left(b\right)$.
It is well known that 
\begin{equation}  \label{eq_ineq proy}
\left\langle\widehat{x}-x,y\right\rangle\ge\left\langle\widehat{x}-x,%
\widehat{x}\right\rangle \text{ for all }\;y\in\mathcal{E}_{\mathcal{F}%
}\left(b\right).
\end{equation}
In particular, for any $y_{0}\in\mathcal{F}\left(b\right)$, all $%
u\in\left\{c_{1},...,c_{q}\right\}^{\circ}$, and all $\lambda>0$ we have $%
\left\langle\widehat{x}-x,y_{0}+\lambda u\right\rangle\ge\left\langle%
\widehat{x}-x,\widehat{x}\right\rangle$. Dividing both sides of the latter
inequality by $\lambda>0$ and letting $\lambda\rightarrow\infty$ give us $%
\left\langle\widehat{x}-x,u\right\rangle\ge 0$, i.e., 
\begin{equation*}
\widehat{x}-x\in\left\{c_{1},...,c_{q}\right\}^{\circ\circ}=\mathrm{cone}%
\left\{c_{1},...,c_{q}\right\}.
\end{equation*}
Thus we have $\widehat{x}-x=\mu c$ for some $c\in\mathrm{conv}%
\left\{c_{1},...,c_{q}\right\}$ and some $\mu>0$ by taking into account that 
$\widehat{x}-x\ne 0_{n}$. To verify now that $x\not\in\mathcal{F}%
\left(b\right)+\left\{c\right\}^{\circ}$, suppose the contrary and then
deduce from the above that $x=y+u$ with some $y\in\mathcal{F}\left( b\right)$
and $u\in\left\{c\right\}^{\circ}$. It tells us that 
\begin{equation*}
\left\langle c,x\right\rangle=\left\langle c,y\right\rangle+\left\langle
c,u\right\rangle\ge\left\langle c,y\right\rangle\ge\left\langle c,\widehat{x}%
\right\rangle>\left\langle c,x\right\rangle,
\end{equation*}
where the penultimate step comes from \eqref{eq_ineq proy}, while the last
one follows from the projection inequality 
\begin{equation*}
\left\langle c,\widehat{x}-x\right\rangle=\frac{1}{\mu}\left\Vert\widehat{x}%
-x\right\Vert_{2}^{2}
\end{equation*}
with $\left\Vert\cdot\right\Vert_{2}$ standing for the Euclidean norm. The
obtained contradiction completes the proof of the theorem.
\end{proof}

\begin{rem}
\emph{Observe that in Theorem~\ref{Th_MOP value} we cannot avoid the convex
combination in the representation of $\mathcal{E}_{\mathcal{F}%
}\left(b\right) $, i.e., replace $\mathrm{conv}\left\{c_{1},...,c_{q}\right%
\} $ by $\left\{c_{1},...,c_{q}\right\}$.} \emph{To illustrate it, consider
the case where} $\mathbb{R}^{T}=\mathbb{R}^{2}$, 
\begin{equation*}
\mathcal{F}\left( b\right)=\mathrm{conv}\left\{\left(1,0\right),\left(0,1%
\right)\right\},\;\emph{and }\;
c_{1}=\left(1,0\right),\;c_{2}=\left(0,1\right).
\end{equation*}
\emph{However, the set }$\mathrm{conv}\left\{c_{1},...,c_{q}\right\}$ \emph{%
can be replaced by any basis of the cone} $C$.
\end{rem}

To establish efficient representations of the sets in the form $\mathcal{F}%
\left(b\right)+\left\{c\right\}^{\circ}$, and hence of $\mathcal{E}_{%
\mathcal{F}}\left(b\right)$ due to Theorem~\ref{Th_MOP value}, we focus now
on multifunctions $\mathcal{F}$ defined as 
\begin{equation*}
\mathcal{F}\left(\cdot\right)+\mathrm{cone}\left\{u\right\}\;\mbox{ and }\;%
\mathcal{F}\left(\cdot\right)+\mathrm{span}\left\{u\right\}.
\end{equation*}
This is done in the remainder of this section.\vspace*{0.05in}

Given $u\in\mathbb{R}^{n}$, consider first the polyhedral set $\mathcal{F}%
\left(b\right)+\mathrm{cone\,}\left\{u\right\}$ and introduce the following
partition $\left\{T_{1},T_{2}\right\}$ of $T=\left\{1,...,m\right\}$: 
\begin{equation}  \label{Eq_T12}
\begin{array}{cc}
T_{1}:=\big\{t\in T\big|\;\left\langle a_{t},u\right\rangle\le 0\big\}\;%
\mbox{ and }\; T_{2}:=\big\{t\in T\big|\;\left\langle a_{t},u\right\rangle>0%
\big\}. & 
\end{array}%
\end{equation}
Then for each $t\in T_{1}$ we denote 
\begin{equation*}
a_{\left(t,0\right)}:=a_{t}\text{ and }b_{\left(t,0\right)}:=b_{t}
\end{equation*}
and for each $\left(t,s\right)\in T_{1}\times T_{2}$ denote 
\begin{equation}  \label{eq_a_ts}
\begin{array}{cc}
a_{\left(t,s\right)}:=\left\langle a_{s},u\right\rangle a_{t}-\left\langle
a_{t},u\right\rangle a_{s}, & b_{\left(t,s\right)}:=\left\langle
a_{s},u\right\rangle b_{t}-\left\langle a_{t},u\right\rangle b_{s}.%
\end{array}%
\end{equation}
With $\widetilde{T}:=T_{1}\times\left(\{0\}\cup T_{2}\right)$ let us now
define the linear inequality system 
\begin{equation}  \label{eq_sigmatilde}
\widetilde{\sigma}\left( b\right):=\left\{\left\langle
a_{\left(t,s\right)},x\right\rangle\le
b_{\left(t,s\right)},~\left(t,s\right)\in\widetilde{T}\right\}
\end{equation}
and denote by $\widetilde{\mathcal{F}}\left(b\right)$ the set of feasible
solutions to $\widetilde{\sigma}\left(b\right)$.

\begin{rem}
\emph{If }$T_{1}=\emptyset $.\emph{\ then} $\widetilde{T}=\emptyset $ \emph{%
\ and }$\widetilde{\mathcal{F}}\left( b\right) =\mathbb{R}^{n}$. \emph{%
Otherwise we have that }$\left( a_{\left( t,s\right) },b_{\left( t,s\right)
}\right) $\emph{\ is a conic combination of }$\left( a_{t},b_{t}\right) $ 
\emph{and } $\left( a_{s},b_{s}\right) $ \emph{for all }$\left( t,s\right)
\in \widetilde{T}$. \emph{It is clear then that }$\mathcal{F}\left( b\right)
\subset \widetilde{\mathcal{F}}\left( b\right) $. \emph{Observe also that,
in contrast to }$\sigma \left( b\right) $, \emph{the new system} $\widetilde{%
\sigma }\left( b\right) $ \emph{is no longer parameterized by its RHS.}
\end{rem}

The next theorem represents $\mathcal{F}\left(b\right)+\mathrm{cone\,}%
\left\{u\right\}$ as the set of feasible solutions to the new linear
inequality system \eqref{eq_sigmatilde}.

\begin{theo}
\label{Prop_system_epi} In terms of the notation above, for any $b\in\mathbb{%
R}^{T}$ we have 
\begin{equation}  \label{eq_ef}
\widetilde{\mathcal{F}}\left(b\right)=\mathcal{F}\left(b\right)+\mathrm{%
cone\,}\left\{u\right\}.
\end{equation}
\end{theo}

\begin{proof}
Let us first verify the inclusion `$\supset$' in \eqref{eq_ef}. Taking $x\in%
\mathcal{F}\left(b\right)+\mathrm{cone\,}\left\{u\right\}$, we get the
linear inequalities 
\begin{equation}  \label{eq_10}
\left\langle a_{t},x-\mu u\right\rangle\le b_{t}\text{ for all }\;t\in T 
\text{ and some }\;\mu\ge 0.
\end{equation}
There is nothing to prove if $T_1=\emptyset$. Otherwise we fix $t\in T_{1}$
and get $\left\langle a_{t},x\right\rangle\le b_{t}$. Taking further $s\in
T_{2}$, we distinguish the following two cases. If $\left\langle
a_{t},u\right\rangle=0$, then the aimed inequality 
\begin{equation*}
\left\langle a_{\left(t,s\right)},x\right\rangle\le b_{\left(t,s\right)}
\end{equation*}
reduces to $\left\langle a_{t},x\right\rangle\le b_{t}$. In the case where $%
\left\langle a_{t},u\right\rangle<0$ we deduce from \eqref{eq_10} that 
\begin{equation*}
\frac{\left\langle a_{s},x\right\rangle-b_{s}}{\left\langle
a_{s},u\right\rangle}\le\mu\le\frac{\left\langle a_{t},x\right\rangle-b_{t}}{%
\left\langle a_{t},u\right\rangle}.
\end{equation*}
In particular, it follows that 
\begin{equation*}
\frac{\left\langle a_{s},x\right\rangle-b_{s}}{\left\langle
a_{s},u\right\rangle}\le\frac{\left\langle a_{t},x\right\rangle-b_{t}}{%
\left\langle a_{t},u\right\rangle},
\end{equation*}
which readily implies that 
\begin{equation*}
\left\langle a_{\left(t,s\right)},x\right\rangle\le b_{\left(t,s\right)}
\end{equation*}
and thus verifies the inclusion `$\supset$' in \eqref{eq_ef}.

To prove now the opposite inequality `$\subset$' in \eqref{eq_ef}, pick any $%
x\in\widetilde{\mathcal{F}}\left(b\right)$ and let us verify the existence
of $\mu\ge 0$ such that 
\begin{equation*}
\left\langle a_{t},x-\mu u\right\rangle\le b_{t}\;\mbox{ for all }\;t\in T.
\end{equation*}
Indeed, when $T_{2}=\emptyset$ we get $x\in\mathcal{F}\left(b\right)$, which
agrees in this case with $\widetilde{\mathcal{F}}\left(b\right)$. If $%
T_{2}\ne\emptyset$, it is sufficient to consider any $\mu\ge 0$ satisfying 
\begin{equation*}
\max_{s\in T_{2}}\frac{\left\langle a_{s},x\right\rangle-b_{s}}{\left\langle
a_{s},u\right\rangle}\le\mu\le\inf_{\left\langle a_{t},u\right\rangle<0}%
\frac{\left\langle a_{t},x\right\rangle-b_{t}}{\left\langle
a_{t},u\right\rangle}
\end{equation*}
under the usual convention that $\inf\emptyset=\infty$. We complete the
proof of the theorem by observing that such a number $\mu$ exists due to the
choice of $x$.
\end{proof}

\vspace*{0.05in}

Looking closely at the proof of Theorem~\ref{Prop_system_epi} tells us that
the successive application of the procedure therein is instrumental to
represent the more general feasible sets $\mathcal{F}\left( b\right) +%
\mathrm{cone\,}\left\{ u_{1},...,u_{p}\right\} $ via linear inequality
systems. However, explicit forms of such representations may generally be
rather complicated. In the next theorem we consider the important case where 
\begin{equation*}
\mathcal{F}\left( b\right) +\mathrm{span\,}\big\{u\big\}
\end{equation*}%
for which we give a direct proof.

\begin{theo}
\label{Prop_F+Ru} Given any $b\in \mathbb{R}^{T}$ and recalling the notation
in \eqref{eq_a_ts}, we have 
\begin{equation}
\mathcal{F}\left( b\right) +\mathrm{span}\big\{u\big\}=\left\{ x\in \mathbb{R%
}^{n}\left\vert 
\begin{array}{ll}
\left\langle a_{t},x\right\rangle \leq b_{t}\text{ if }\left\langle
a_{t},u\right\rangle =0, &  \\ 
\left\langle a_{\left( t,s\right) },x\right\rangle \leq b_{\left( t,s\right)
}\text{ if}\;\left\langle a_{t}u\right\rangle <0 &  \\ 
\qquad \qquad \qquad \quad \;\;\mbox{and}\,\left\langle a_{s},u\right\rangle
>0 & 
\end{array}%
\right. \right\} .  \label{eq_F+span}
\end{equation}
\end{theo}

\begin{proof}
Let us introduce the index set%
\begin{equation*}
T_{0}:=\{t\in T\mid \left\langle a_{t},u\right\rangle =0\}\text{(}\subset
T_{1}\text{).}
\end{equation*}%
Note that, in the case $T_{1}=\emptyset $ (or $T_{2}=T_{0}=\emptyset ),$ the
system (\ref{eq_F+span}) has no inequality (i.e., its solution set is the
whole space $\mathbb{R}^{n}$), but in such cases $\mathcal{F}\left( b\right)
+\mathrm{span\,}\left\{ u\right\} $ is also $\mathbb{R}^{n}$, as $-u\in 
\limfunc{int}(O^{+}\mathcal{F}\left( b\right) )$ (or $u\in \limfunc{int}%
(O^{+}\mathcal{F}\left( b\right) )$, respectively), entailing that if $%
x_{0}\in \mathcal{F}\left( b\right) ,$ 
\begin{equation*}
x_{0}+\mathrm{span\,}\left\{ u\right\} +\lambda \mathbb{B\subset \ }\mathcal{%
F}\left( b\right) +\mathrm{span\,}\left\{ u\right\} ,\text{ for all }\lambda
\geq 0.
\end{equation*}%
If $T_{1}\diagdown T_{0}$ and $T_{2}$ are both nonempty, the reasoning is
the same followed in Proposition \ref{Prop_system_epi}, without taking into
account the sign of $\mu .$
\end{proof}

\vspace*{0.05in}

The reader will see in Example~\ref{Exa_new} below a detailed illustration
of both Theorems~\ref{Prop_system_epi} and \ref{Prop_F+Ru} together with
additional comments on the relationship between the optimal value and the
epigraphical mapping $\mathcal{E}_{\mathcal{F}}$ in linear programming.

\section{Subdifferentials of Epigraphical Pareto Fronts}

\label{pareto}

This section concerns the epigraphical Pareto front multifunction $\mathcal{E%
}_{\mathcal{P}}\colon\mathbb{R}^{T}\mathbb{\rightrightarrows R}^{q}$
introduced in \eqref{eq_epi_Pareto_front} in the form 
\begin{equation*}
\mathcal{E}_{\mathcal{P}}\left(b\right):=\mathcal{P}\left(b\right)+\mathbb{R}%
_{+}^{q},\quad b\in\mathbb{R}^{T},
\end{equation*}%
where the Pareto front mapping $\mathcal{P}$ is defined in \eqref{eq_pareto}%
. In contrast to $\mathrm{gph}\mathcal{F}$, the set $\mathrm{gph\mathcal{P}}$
is nonconvex in general, while the one of our interest $\mathrm{gph\mathrm{%
\mathcal{E}}_{\mathrm{\mathcal{P}}}}$ is \emph{always convex}. This is shown
in the next proposition.

\begin{prop}
\label{Prop_EP_convex} The set $\mathrm{gph\mathrm{\mathcal{E}}_{\mathrm{%
\mathcal{P}}}}$ is a closed and convex subset of $\mathbb{R}^{T}\times%
\mathbb{R}^{q}$.
\end{prop}

\begin{proof}
First we observe that the set $\mathrm{gph\mathcal{P}}$ is a finite union of
convex polyhedral cones as the KKT (or primal/dual) optimality conditions in
linear programming allow us to express $\mathrm{gph\mathcal{P}}$ as the
graph of a certain feasible set mapping of a linear system and we can apply
then the classical result by Robinson \cite{Robinson}. Hence \emph{a fortiori%
} the set $\mathrm{gph\mathrm{\mathcal{E}}_{\mathrm{\mathcal{P}}}}=\mathrm{%
gph\mathcal{P+}}\left( \left\{ 0_{m}\right\} \times \mathbb{R}%
_{+}^{q}\right) $ is also closed.

Let us now show that the set $\mathrm{gph\mathrm{\mathcal{E}}_{\mathrm{%
\mathcal{P}}}}$ is convex. Fix any two pairs $\left(b_{1},p_{1}+u_{1}\right)$%
, $\left(b_{2},p_{2}+u_{2}\right)\in\mathrm{gph\mathrm{\mathcal{E}}_{\mathrm{%
\mathcal{P}}}}$, i.e., such that $b_{i}\in\mathbb{R}^{T}$, $p_{i}\in\mathcal{%
P}\left(b_{i}\right)$, and $u_{i}\in\mathbb{R}_{+}^{q}$ as $i=1,2$. Then for
every $\lambda\in\lbrack 0,1]$ we have 
\begin{equation*}
p_{i}=\left(\left\langle c_{1},x_{i}\right\rangle,...,\left\langle
c_{q},x_{i}\right\rangle\right) \text{ with some }\;x_{i}\in\mathcal{S}%
\left(b_{i}\right),\text{ }i=1,2,
\end{equation*}
and so $\left(1-\lambda\right)x_{1}+\lambda x_{2}\in\mathcal{F}%
\left(\left(1-\lambda\right)b_{1}+\lambda b_{2}\right)$. In the nontrivial
case where 
\begin{equation*}
\left(1-\lambda\right)x_{1}+\lambda x_{2}\notin\mathcal{S}%
\left(\left(1-\lambda\right)b_{1}+\lambda b_{2}\right)
\end{equation*}
we apply Lemma~\ref{Lem_F_notS} to get the existence of $\widetilde{x}\in%
\mathcal{S}\left(\left(1-\lambda\right) b_{1}+\lambda b_{2}\right)$ with $%
\left\langle c_{i},\widetilde{x}\right\rangle\le\left\langle
c_{i},\left(1-\lambda\right)x_{1}+\lambda x_{2}\right\rangle$ for all $%
c_1,\ldots,c_q$. It implies that 
\begin{equation*}
\left(1-\lambda\right)p_{1}+\lambda p_{2}\in\left(\left\langle c_{1}, 
\widetilde{x}\right\rangle ,...,\left\langle c_{q},\widetilde{x}%
\right\rangle\right)+\mathbb{R}_{+}^{q},
\end{equation*}
which can be equivalently written as 
\begin{equation*}
\left(1-\lambda\right) p_{1}+\lambda p_{2}\in\mathcal{E}_{\mathcal{P}%
}\left(\left(1-\lambda\right)b_{1}+\lambda b_{2}\right).
\end{equation*}
Therefore, we arrived at the inclusion 
\begin{equation*}
\left(1-\lambda\right)\left(p_{1}+u_{1}\right)+\lambda\left(p_{2}+u_{2}%
\right)\in\mathcal{E}_{\mathcal{P}}\left(\left(1-\lambda\right)b_{1}+\lambda
b_{2}\right),
\end{equation*}
which verifies the convexity of the set $\mathrm{gph\mathrm{\mathcal{E}}_{%
\mathrm{\mathcal{P}}}}$ .
\end{proof}

\vspace*{0.05in}

Using the above proposition and employing the fundamental results of Theorem~%
\ref{Cor_lip_subdiff}, we can now conduct a local stability analysis of the
epigraphical Pareto front mapping similarly to that for the epigraphical
feasible solution mapping developed in Section~\ref{feas}.

\begin{theo}
\label{Cor_lip_EP} Let $\left( \overline{b},\overline{p}\right) \in \mathrm{%
gph\mathcal{P}}$. Then we have 
\begin{equation*}
\partial \mathcal{P}\left( \overline{b},\overline{p}\right) =\dbigcup_{%
\QATOP{\scriptstyle{\alpha \in \mathbb{R}_{+}^{q}}}{\scriptstyle{\Vert
\alpha \Vert }_{\ast }{=1}}}\big\{y\in \mathbb{R}^{T}\big|\;\left\langle y,b-%
\overline{b}\right\rangle \leq \left\langle \alpha ,p-\overline{p}%
\right\rangle \text{ for all }\left( b,p\right) \in \mathrm{gph}\mathcal{P}%
\big\}.
\end{equation*}%
Furthermore, the Lipschitz modulus of the epigraphical Pareto front mapping $%
\mathcal{E}_{\mathcal{P}}$ at $\left( \overline{b},\overline{x}\right) $ is
computed by the formula 
\begin{equation*}
\mathrm{lip}\mathcal{E}_{\mathcal{P}}\left( \overline{b},\overline{x}\right)
=\sup \left\{ \{\left\Vert y\right\Vert _{1}\big|\;y\in \partial \mathcal{P}%
\left( \overline{b},\overline{p}\right) \right\} ,
\end{equation*}%
which ensures that the mapping $\mathcal{E}_{\mathcal{P}}$ is Lipschitz-like
around $\left( \overline{b},\overline{p}\right) $ if and only if 
\begin{equation*}
\sup \Big\{\{\left\Vert y\right\Vert _{1}\Big|\;\dbigcup_{\QATOP{\scriptstyle%
{\alpha \in \mathbb{R}_{+}^{q}}}{\scriptstyle{\Vert \alpha \Vert }_{\ast }{=1%
}}}\big\{y\in \mathbb{R}^{T}\big|\;\left\langle y,b-\overline{b}%
\right\rangle \leq \left\langle \alpha ,p-\overline{p}\right\rangle \text{%
for all}\;\left( b,p\right) \in \mathrm{gph}\mathcal{P}\big\}\Big\}<\infty .
\end{equation*}
\end{theo}

\begin{proof}
Having in hand Proposition~\ref{Prop_EP_convex}, we can apply Theorem~\ref%
{Cor_lip_subdiff} and then proceed similarly to the proofs of Theorems~\ref%
{Cor_lip_subdiff} and \ref{Cor_lip_EF}.
\end{proof}

\vspace*{0.05in}

The next result expresses the subdifferential $\partial\mathcal{F}\left(%
\overline{b},\overline{x}\right)$ in terms of $\mathrm{gph}\mathcal{P}$
instead of $\mathrm{gph}\mathcal{S}$. Observe that the difference between
the expression for $\partial\mathcal{F}\left(\overline{b},\overline{x}%
\right) $ obtained below and the one for $\partial\mathcal{P}\left(\overline{%
b},\overline{p}\right)$ established in Theorem~\ref{Cor_lip_EP} is seen only
in the sets where the vector $\alpha\in\mathbb{R}_{+}^{q}$ takes its values.

\begin{theo}
\label{Prop_subdiffF} Let $\left( \overline{b},\overline{x}\right) \in 
\mathrm{gph\mathcal{S}}$, and let $\overline{p}=\left( \left\langle c_{1},%
\overline{x}\right\rangle ,...,\left\langle c_{q},\overline{x}\right\rangle
\right) \in \mathrm{\mathcal{P}}\left( \overline{b}\right) $. Then the
subdifferential of $\mathcal{F}$ at $\left( \overline{b},\overline{x}\right) 
$ is computed by 
\begin{equation*}
\partial \mathcal{F}\left( \overline{b},\overline{x}\right) =\bigcup_{\QATOP{%
\scriptstyle{\alpha \in \mathbb{R}_{+}^{q}}}{\scriptstyle{\left\Vert \dsum
\alpha _{i}c_{i}\right\Vert }_{\ast }{=1}}}\big\{y\in \mathbb{R}^{T}\big|%
\;\left\langle y,b-\overline{b}\right\rangle \leq \left\langle \alpha ,p-%
\overline{p}\right\rangle \;\text{ for all }\;\left( b,p\right) \in \mathrm{%
gph}\mathcal{P}\big\}.
\end{equation*}
\end{theo}

\begin{proof}
Taking into account the previous considerations, we proceed similarly to the
proof of Theorem~\ref{Theo_subdiffF}.
\end{proof}

\vspace*{0.05in}

Let us now present a two-dimensional numerical example that illustrates how
both Theorems~\ref{Cor_lip_EP} and \ref{Prop_subdiffF} can be applied in
computation.

\begin{exam}
\label{Exa1} \emph{Take any} $a>0$ \emph{and consider the following
multiobjective problem} \eqref{eq_MOP} \emph{with $n=q=2$ and the Euclidean
norms on both spaces}: 
\begin{equation*}
\begin{tabular}{ccc}
$MLP\left( b\right) :$ & \textrm{minimize} & $\left( ax_{1},x_{2}\right) $
\\ 
& subject to & $x_{1}\geq b_{1}$, \\ 
&  & $x_{2}\geq b_{2}$.%
\end{tabular}%
\end{equation*}%
\emph{Letting} $\overline{b}=0_{2}$, \emph{we easily see that} 
\begin{align*}
& \mathcal{F}\left( b\right) =\mathcal{E}_{\mathcal{F}}\left( b\right)
=\{\left( b_{1},b_{2}\right) \}+\mathbb{R}_{+}^{2},\;\mathcal{S}\left(
b\right) =\text{ }\{\left( b_{1},b_{2}\right) \},\text{ } \\
& \mathcal{P}\left( b\right) =\text{ }\{\left( ab_{1},b_{2}\right) \},\;%
\text{\textrm{and }}\mathcal{E}_{\mathcal{P}}\left( b\right) =\{\left(
ab_{1},b_{2}\right) \}+\mathbb{R}_{+}^{2}.
\end{align*}%
\emph{Furthermore, using Theorems~\ref{Theo_subdiffF} and \ref{Cor_lip_EP}
tells us, respectively, that} 
\begin{eqnarray*}
\partial \mathcal{F}\left( 0_{2},0_{2}\right) &=&\bigcup_{\QATOP{\scriptstyle%
{\alpha ^{2}\alpha _{1}^{2}+\alpha _{2}^{2}=1}}{\scriptstyle{\alpha
_{1},\alpha _{2}\geq 0}}}\big\{y\in \mathbb{R}^{2}\big|%
\;b_{1}y_{1}+b_{2}y_{2}\leq a\alpha _{1}b_{1}+\alpha _{2}b_{2}\;\text{ }%
\forall \;b_{1},b_{2}\in \mathbb{R}\big\} \\
&=&\Big\{y\in \mathbb{R}^{2}\Big|\;y_{1}^{2}+y_{2}^{2}=1,\;y_{1},y_{2}\geq 0%
\Big\},
\end{eqnarray*}%
\begin{eqnarray*}
\partial \mathcal{P}\left( 0_{2},0_{2}\right) &=&\bigcup_{\QATOP{\scriptstyle%
{\alpha _{1}^{2}+\alpha _{2}^{2}=1}}{\scriptstyle{\alpha _{1},\alpha
_{2}\geq 0}}}\Big\{y\in \mathbb{R}^{2}\Big|\;b_{1}y_{1}+b_{2}y_{2}\leq
a\alpha _{1}b_{1}+\alpha _{2}b_{2}\text{ }\forall \;b_{1},b_{2}\in \mathbb{R}%
\Big\} \\
&=&\Big\{\left( y_{1},y_{2}\right) \in \mathbb{R}^{2}\Big|\;\frac{y_{1}^{2}}{%
a^{2}}+y_{2}^{2}=1,\text{ }y_{1},y_{2}\geq 0\Big\}.
\end{eqnarray*}%
\emph{Since $\mathcal{F}=\mathcal{E}_{\mathcal{F}}$ in this case,} \emph{we
can appeal to Theorem~\ref{Cor_lip_EF} (cf.\ also \cite[Corollary~3.2]%
{CDLP05}) to compute the Lipschitz modulus:} 
\begin{equation*}
\mathrm{lip}\mathcal{F}\left( 0_{2},0_{2}\right) =\mathrm{lip}\mathcal{E}_{%
\mathcal{F}}\left( 0_{2},0_{2}\right) =\sqrt{2}.
\end{equation*}%
\emph{Considering now the mappings} $\mathcal{P}$\emph{\ and }$\mathcal{E}_{%
\mathcal{P}}$, \emph{we can treat them as the feasible set mappings for the
equality and inequality systems with respect to the variables} $\left(
p_{1},p_{2}\right) $. \emph{Namely, as $(1/a)p_{1}=b_{1},p_{2}=b_{2}$}\emph{%
\ and }$(1/a)p_{1}\geq b_{1},p_{2}\geq b_{2}$, \emph{respectively. Appealing
to Theorem~\ref{Cor_lip_EP} (cf.\ also \cite[Corollary~3.2]{CDLP05}), we
obtain} 
\begin{equation*}
\mathrm{lip}\mathcal{P}\left( 0_{2},0_{2}\right) =\mathrm{lip}\mathcal{E}_{%
\mathcal{P}}\left( 0_{2},0_{2}\right) =\sqrt{1+a^{2}}.
\end{equation*}
\end{exam}

\vspace*{0.07in}

As follows from \eqref{eq_lipEmenorlip}, we have 
\begin{equation}  \label{ex-ine}
\mathrm{lip}\mathcal{P}\left(\overline{b},\overline{p}\right)\ge\mathrm{lip}%
\mathcal{E}_{\mathcal{P}}\left(\overline{b},\overline{p}\right), \text{ for
any }\left(\overline{b},\overline{p}\right)\in\mathrm{gph\mathcal{P}}.
\end{equation}%
\vspace*{0.01in}

The next example shows that the inequality in \eqref{ex-ine} may be \emph{%
strict}.

\begin{exam}
\label{Exa2} \emph{Consider the multiobjective problem} \eqref{eq_MOP} \emph{%
with $n=q=2$ and the Euclidean norms on both spaces}: 
\begin{equation*}
\begin{tabular}{ccc}
$MLP\left( b\right) :$ & \textrm{minimize} & $\left( x_{1},x_{2}\right) $ \\ 
& subject to & $x_{1}\geq b_{1}$, \\ 
&  & $x_{1}+x_{2}\geq b_{2}$.%
\end{tabular}%
\end{equation*}%
\emph{It is easy to check that }$\mathcal{F}\left( b\right) =\mathcal{E}_{%
\mathcal{F}}\left( b\right) $ for any $b\in \mathbb{R}^{2}$ \emph{and that} 
\begin{equation*}
\mathrm{lip}\mathcal{F}\left( 0_{2},0_{2}\right) =\mathrm{lip}\mathcal{E}_{%
\mathcal{F}}\left( 0_{2},0_{2}\right) =1.
\end{equation*}%
\emph{On the other hand, we clearly have the expressions} 
\begin{align*}
\mathcal{P}\left( b\right) & =\big\{p\in \mathbb{R}^{2}\big|\;p_{1}\geq
b_{1},p_{1}+p_{2}=b_{2}\},\text{ }b\in \mathbb{R}^{2}, \\
\mathcal{E}_{\mathcal{P}}\left( b\right) & =\{p\in \mathbb{R}^{2}\big|%
\;p_{1}\geq b_{1},p_{1}+p_{2}\geq b_{2}\},\text{ }b\in \mathbb{R}^{2},
\end{align*}%
\emph{with the strict inequality} 
\begin{equation*}
\mathrm{lip}\mathcal{P}\left( 0_{2},0_{2}\right) =\sqrt{5}>\mathrm{lip}%
\mathcal{E}_{\mathcal{P}}\left( 0_{2},0_{2}\right) =1.
\end{equation*}
\end{exam}

\vspace*{0.05in}

Observe that the situation of Example~\ref{Exa2} does not occur in the case
of single-objective linear programs, where we always have $\mathrm{lip}%
\mathcal{P}\left(\overline{b},\overline{p}\right)=\mathrm{lip}\mathcal{E}_{%
\mathcal{P}}\left(\overline{b},\overline{p}\right)$ with $\left(\overline{b},%
\overline{p}\right)\in\mathrm{gph\mathcal{P}}$. This is one of the main
points of the next section.

\section{Lipschitz Moduli in Linear Programming}

\label{lp}

In this section we provide specifications and further developments of the
results obtained above for the general linear multiobjective problem %
\eqref{eq_MOP} for the case of ordinary linear programs given by 
\begin{equation}  \label{eq_LP}
\begin{tabular}{ccc}
$LP\left(b\right):$ & \emph{\textrm{minimize}} & $\left\langle
c,x\right\rangle$ \\ 
& \emph{subject to} & $x\in\mathcal{F}\left(b\right)$%
\end{tabular}%
\end{equation}
where the vector $c\in\mathbb{R}^{n}$ is fixed. As shown below, the approach
and results developed for linear multiobjective problems lead us to refined
computation formulas for subdifferentials and Lipschitz moduli in standard
problems of linear programming under parameter perturbations.

In what follows we assume that $-c\in \mathrm{cone}\{a_{t},t\in T\}$ (dual
consistency); otherwise the problem $LP\left( b\right) $ is always
unsolvable. Denote by $\vartheta \colon \mathbb{R}^{T}\rightarrow \mathbb{R}%
\cup \{+\infty \}$ the associated \emph{optimal value function} defined by 
\begin{equation*}
\vartheta \left( b\right) :=\inf \big\{\left\langle c,x\right\rangle \big|%
\;x\in \mathcal{F}\left( b\right) \big\}.
\end{equation*}%
We can easily see in this framework that 
\begin{equation*}
\mathrm{dom}\mathcal{\vartheta }:=\big\{b\in \mathbb{R}^{T}\big|\;\vartheta
\left( b\right) <\infty )=\mathrm{dom}\mathcal{F}.
\end{equation*}%
Indeed, it is well known in linear programming that the boundedness of $%
LP\left( b\right) $ is equivalent to its solvability, which is in turn
equivalent to the simultaneous fulfilment of primal and dual consistency.

Observe that in the setting of \eqref{eq_LP} the multifunctions $\mathcal{S}$%
, $\mathcal{P}$, $\mathcal{E}_{\mathcal{P}}$, and $\mathcal{E}_{\mathcal{F}}$
admit the following specifications. For each $b\in\mathbb{R}^{T}$ we have
that $\mathcal{S}\left(b\right)$ is the set of optimal solutions to $%
LP\left(b\right)$ while the mapping $\mathcal{P}\colon\mathbb{R}%
^{T}\rightrightarrows\mathbb{R}$ is actually single-valued given by 
\begin{equation*}
\mathcal{P}\left(b\right)=\left\{%
\begin{array}{l}
\left\{\vartheta\left(b\right)\right\} \text{ if }\;b\in\mathrm{dom}\mathcal{%
S}, \\ 
\emptyset \text{ if }\;b\notin\mathrm{dom}\mathcal{S}.%
\end{array}%
\right.
\end{equation*}
Furthermore, we get the relationships 
\begin{equation*}
\mathrm{dom}\mathcal{P}=\mathrm{dom}\mathcal{S}=\mathrm{dom}\mathcal{F}=%
\mathrm{dom}\mathcal{\vartheta}.
\end{equation*}
Taking into account that $\mathcal{P}\left(b\right)$ is a singleton for any $%
b\in\mathrm{dom}\mathcal{S}$, from now on we write $\partial\mathcal{P}%
\left(b\right)$ instead of $\partial\mathcal{P}\left(b,\vartheta\left(b%
\right)\right)$. Moreover, for each such $b$ it follows that 
\begin{equation*}
\mathcal{E}_{\mathcal{P}}\left(b\right)=[\vartheta\left(b\right),\infty),
\end{equation*}
and it is easily to verify that 
\begin{equation}  \label{eq_F+co geq v}
\mathcal{E}_{\mathcal{F}}\left(b\right):=\mathcal{F}\left(
b\right)+\left\{c\right\}^{\circ }=\big\{x\in\mathbb{R}^{n}\big|%
\;\left\langle c,x\right\rangle\ge \vartheta\left(b\right)\big\}.
\end{equation}
It allows us to show below that Lipschitzian behavior of $\mathcal{E}_{%
\mathcal{F}}$ is closely related to that of $\mathcal{P};$ see Theorem~\ref%
{Th lipV}. To proceed, we first present the following proposition.

\begin{prop}
\label{Prop_sect5} For any $\left( \overline{b},\overline{x}\right) \in 
\mathrm{gph}\mathcal{S}$ we have 
\begin{equation*}
\partial \mathcal{P}\left( \overline{b}\right) =\left\Vert c\right\Vert
_{\ast }\partial \mathcal{F}\left( \overline{b},\overline{x}\right) .
\end{equation*}
\end{prop}

\begin{proof}
It follows from Theorems~\ref{Cor_lip_EP} and \ref{Prop_subdiffF} that 
\begin{align*}
\partial \mathcal{P}\left( \overline{b}\right) & \big\{y\in \mathbb{R}^{T}%
\big|\;\left\langle y,b-\overline{b}\right\rangle \leq p-\overline{p}\text{ }%
\forall \left( b,p\right) \in \mathrm{gph}\mathcal{P}\big\} \\
& =\left\Vert c\right\Vert _{\ast }\Big\{y\in \mathbb{R}^{T}\Big|%
\;\left\langle y,b-\overline{b}\right\rangle \leq \frac{1}{\left\Vert
c\right\Vert _{\ast }}\left( p-\overline{p}\right) \text{ }\forall \left(
b,p\right) \in \mathrm{gph}\mathcal{P}\Big\} \\
& =\left\Vert c\right\Vert _{\ast }\partial \mathcal{F}\left( \overline{b}%
\overline{x}\right) ,
\end{align*}%
which therefore verifies the claimed equality.
\end{proof}

\begin{rem}
\label{Rem_neigh} \emph{Given $\overline{b}\in\mathrm{dom}\mathcal{P}$} 
\emph{and remembering that }$\mathcal{P}\left(b\right)$\emph{\ is a
singleton for any }$b\in\mathrm{dom}\mathcal{P}$, \emph{we can write} 
\begin{equation*}
\partial\mathcal{P}\left(\overline{b}\right)=\big\{y\in\mathbb{R}^{T}\big|%
\;\left\langle y,b-\overline{b}\right\rangle\le\vartheta\left(b\right)-%
\vartheta\left(\overline{b}\right) \text{ }\forall\;b\in\mathrm{dom}\mathcal{%
P}\big\},
\end{equation*}
\emph{which agrees with the classical subdifferential of $\mathcal{P}$ at $%
\overline{b}$ in the sense of convex analysis}. \emph{It is actually not
surprising since the convexity of the function }$\vartheta$ \emph{can be
clearly derived from Proposition~\ref{Prop_EP_convex}. Going a little
further, observe that the set} $\mathrm{dom}\mathcal{F}$ \emph{can be
replaced by any intersection of the form} $\mathrm{dom}\mathcal{F\cap U}_{%
\overline{b}}$, \emph{where }$\mathcal{U}_{\overline{b}}\in\mathbb{R}^{T}$ 
\emph{is an arbitrary neighborhood of }$\overline{b}$.
\end{rem}

Now we are ready to formulate and prove the last theorem of this paper.

\begin{theo}
\label{Th lipV} Let $\left( \overline{b},\overline{x}\right) \in \mathrm{gph}%
\mathcal{S}$. Then 
\begin{equation*}
\mathrm{lip}\mathcal{P}\left( \overline{b}\right) =\mathrm{lip}\mathcal{E}_{%
\mathcal{P}}\left( \overline{b}\right) =\left\Vert c\right\Vert _{\ast }%
\mathrm{lip}\mathcal{E}_{\mathcal{F}}\left( \overline{b},\overline{x}\right)
.
\end{equation*}
\end{theo}

\begin{proof}
The first equality is standard since $\mathcal{P}$ is single-valued on $%
\mathbb{R}$. The second equality follows from Theorem~\ref{Cor_lip_EP} and
Proposition~\ref{Prop_sect5} with taking into account the fact that the
Lipschitz moduli under consideration agree with the suprema of the norms in
the corresponding subdifferentials.
\end{proof}

\vspace*{0.05in}

The next example shows how the obtained results are applied in the case of
two-dimensional linear programming with multiply inequality constraints.

\begin{exam}
\label{Exa_new} \emph{Consider the following parameterized linear program in
the space }$\mathbb{R}^{2}$\emph{\ with the Euclidean norm on it}: 
\begin{equation*}
\begin{tabular}{ccc}
$PL\left( b\right) :$ & \textrm{minimize} & $2x_{1}+x_{2}$ \\ 
& subject to & \multicolumn{1}{r}{$-x_{1}-x_{2}\leq b_{1}$} \\ 
&  & \multicolumn{1}{r}{$-x_{1}+2x_{2}\leq b_{2}$} \\ 
&  & \multicolumn{1}{r}{$-2x_{1}\leq b_{3}$} \\ 
&  & \multicolumn{1}{r}{$3x_{1}+x_{2}\leq b_{4}$}%
\end{tabular}%
\end{equation*}%
\emph{around the nominal parameter }$\overline{b}=\left( -2,1,-2,7\right) $. 
\emph{Since} 
\begin{equation*}
\mathcal{F}\left( b\right) +\left\{ c\right\} ^{\circ }=\mathcal{F}\left(
b\right) +\mathrm{cone}\left\{ \left( 2,1\right) \right\} +\mathrm{span\,}%
\left\{ \left( -1,2\right) \right\} ,
\end{equation*}%
\emph{we first apply Theorem~\ref{Prop_system_epi} with }$u=\left(
2,1\right) $. \emph{Recalling the notation above tells us that} $%
T_{1}=\{1,2,3\},\;T_{2}=\{4\}$, \emph{and} 
\begin{equation*}
\widetilde{\sigma }\left( b\right) =\left\{ 
\begin{array}{l}
-x_{1}-x_{2}\leq b_{1},~-x_{1}+2x_{2}\leq b_{2},~-2x_{1}\leq b_{3}, \\ 
2x_{1}-4x_{2}\leq 7b_{1}+3b_{4},~-7x_{1}+14x_{2}\leq 7b_{2}, \\ 
-2x_{1}+4x_{2}\leq 7b_{3}+4b_{4}%
\end{array}%
\right\} .
\end{equation*}%
\emph{Note that the fifth constraint is equivalent to the second one for all 
}$b\in \mathbb{R}^{\widetilde{T}}$, \emph{while the sixth constraint is
redundant at }$\overline{b}$ \emph{but not at any }$b$. \emph{Remark~\ref%
{Rem_neigh} implies that the sixth inequality is irrelevant in a local
analysis around }$\overline{b}$. \emph{Anyway, let us remove just the fifth
inequality and renumber the resulting }$\widetilde{T}$ \emph{as} $%
\{1,...,5\}.$ \emph{Then apply Theorem~\ref{Prop_F+Ru} to the reduced and
renumbered system }$\widetilde{\sigma }$ \emph{with }$\widetilde{u}=\left(
-1,2\right) $ \emph{to obtain }$\left\langle \widetilde{a}_{t},\widetilde{u}%
\right\rangle <0$ \emph{for }$t\in \{1,4\}$ \emph{and }$\left\langle 
\widetilde{a}_{s},\widetilde{u}\right\rangle >0$ \emph{for }$s\in \{2,3,5\}$%
. \emph{It gives us } 
\begin{equation*}
\widetilde{\widetilde{\sigma }}\left( b\right) =\left\{ 
\begin{array}{c}
-6x_{1}-3x_{2}\leq 5b_{1}+b_{2},~-4x_{1}-2x_{2}\leq 2b_{1}+b_{3}, \\ 
-12x_{1}-6x_{2}\leq 10b_{1}+7b_{3}+4b_{4},~0\leq 35b_{1}+10b_{2}+15b_{4}, \\ 
-16x_{1}-8x_{2}\leq 14b_{1}+10b_{3}+6b_{4},~0\leq 70b_{1}+70b_{3}+70b_{4}%
\end{array}%
\right\} .
\end{equation*}%
\emph{Hence for any} $b\in \mathbb{R}^{4}$ \emph{the system} $\widetilde{%
\widetilde{\sigma }}\left( b\right) $\emph{\ is equivalent to the single
inequality} 
\begin{equation*}
2x_{1}+x_{2}\geq \max \left\{ -\frac{5b_{1}+b_{2}}{3},-\frac{2b_{1}+b_{3}}{2}%
,-\frac{10b_{1}+7b_{3}+4b_{4}}{6},-\frac{7b_{1}+5b_{3}+3b_{4}}{4}\right\}
\end{equation*}%
\emph{provided that }$\min \left\{
7b_{1}+2b_{2}+3b_{4},b_{1}+b_{3}+b_{4}\right\} \geq 0$, \emph{while
otherwise the system} $\widetilde{\widetilde{\sigma }}\left( b\right) $\emph{%
\ is infeasible. Since the right-hand side in} $\widetilde{\widetilde{\sigma 
}}\left( b\right) $ \emph{is } $\left( -9,-6,-6,45,-6,210\right) $, \emph{%
for any }$b$ \emph{close to $\overline{b}$ we have} $\widetilde{\widetilde{%
\sigma }}(b)\equiv 2x_{1}+x_{2}\geq 3$ \emph{and} 
\begin{equation*}
\widetilde{\widetilde{\sigma }}\left( b\right) \equiv 2x_{1}+x_{2}\geq \max
\left\{ -\frac{5b_{1}+b_{2}}{3},-\frac{2b_{1}+b_{3}}{2}\right\} .
\end{equation*}%
\emph{This readily implies for such $b$ that} 
\begin{equation*}
\mathcal{P}\left( b\right) =\max \left\{ -\frac{5b_{1}+b_{2}}{3},-\frac{%
2b_{1}+b_{3}}{2}\right\} .
\end{equation*}%
\emph{Employing now Remark~\ref{Rem_neigh} and the classical formula of
convex analysis for subdifferentiation of maximum functions gives us} 
\begin{equation*}
\partial \mathcal{P}\left( \overline{b}\right) =\mathrm{conv}\big\{\left(
-5/3,-1/3,0,0\right) ,\left( -1,0,-1/2,0\right) \big\}.
\end{equation*}%
\emph{Then we deduce from Theorem~\ref{Cor_lip_EP} that } 
\begin{equation*}
\limfunc{lip}\mathcal{E}_{\mathcal{P}}\left( \overline{b}\right) =\left\Vert
\left( -5/3,-1/3,0,0\right) \right\Vert _{1}=2,
\end{equation*}%
\emph{which ensures in turn by using Theorem~\ref{Th lipV} that} 
\begin{equation*}
\limfunc{lip}\mathcal{P}\left( \overline{b}\right) =2\text{\emph{\ and} }%
\mathrm{lip}\mathcal{E}_{\mathcal{F}}\left( \overline{b},\overline{x}\right)
=2/\sqrt{5}
\end{equation*}%
\emph{at any optimal solution }$\overline{x}$ \emph{of} $PL\left( \overline{b%
}\right) $.
\end{exam}

\begin{rem}
\emph{Paper \cite{GCPT19} provides an alternative way to compute the
Lipschitz modulus} $\limfunc{lip}\mathcal{P}\left( \overline{b}\right) $ 
\emph{under the additional assumption that at least one optimal solution }$%
\overline{x}$ \emph{of} $PL\left( \overline{b}\right) $ \emph{is known. As
we see, the procedure described in Example~\ref{Exa_new} does not require
such an }a priori\emph{\ information.}
\end{rem}

\section{Concluding Remarks}

\label{conc}

This paper demonstrates that employing appropriate tools of variational
analysis and generalized differentiation of set-valued mappings allows us to
efficiently deal with major sensitivity characteristics of perturbed linear
multiobjective optimization problems. Namely, in this way we explicitly
computed the subdifferentials of the feasible set and Pareto front mappings
in such problems together with the exact moduli of their Lipschitzian
stability.

In future research we plan to extend the variational approach and results
obtained in this paper to \emph{convex} problems of multiobjective
optimization by reducing them to linear systems with \emph{block
perturbations}. Observe that a similar procedure has been explored for
feasibility mappings in \emph{semi-infinite} programming with both decision
and parameter variables living in Banach spaces.


\begin{thebibliography}{99}
\bibitem{BaMo07} T. Q. BAO and B. S. MORDUKHOVICH, \emph{Variational
principles for set-valued mappings with applications to multiobjective
optimization}, Control and Cybernetics \textbf{36} (2007), 531--562.

\bibitem{BaMo10} T. Q. BAO and B. S. MORDUKHOVICH, \emph{Relative Pareto
minimizers for multiobjective problems: existence and optimality conditions}%
, Math. Program. \textbf{122} (2010), 301--347.

\bibitem{bz05} J. M. BORWEIN and Q. J. ZHU, \emph{Techniques of Variational
Analysis}, Springer, New York, 2005.

\bibitem{CDLP05} M. J. C\'{A}NOVAS, A. L. DONTCHEV, M. A. L\'{O}PEZ and J.
PARRA, \emph{Metric regularity of semi-infinite constraint systems}, Math.
Program. \textbf{104} (2005), 329--346.

\bibitem{CGP08} M. J. C\'{A}NOVAS, F. J. G\'{O}MEZ-SENENT and J. PARRA, 
\emph{On the Lipschitz modulus of the argmin mapping in linear semi-infinite
optimization}, Set-Valued Anal. \textbf{16} (2008), 511--538.

\bibitem{CKLP07} M. J. C\'{A}NOVAS, D. KLATTE, M. A. L\'{O}PEZ and J. PARRA, 
\emph{Metric regularity in convex semi-infinite optimization under canonical
perturbations}, SIAM J. Optim \textbf{18} (2007), 717--732.

\bibitem{clmp09} M. J. C\'{A}NOVAS, M. A. L\'{O}PEZ, B. S. MORDUKHOVICH and
J. PARRA, \emph{Variational analysis in semi-infinite and infinite
programming, I: Stability of linear inequality systems of feasible solutions}%
, SIAM J. Optim. \textbf{20} (2009), 1504--1526.

\bibitem{clmp12} M. J. C\'{A}NOVAS, M. A. L\'{O}PEZ, B. S. MORDUKHOVICH and
J. PARRA, Quantitative stability of linear infinite inequality systems under
block perturbations with applications to convex systems, \emph{TOP} \textbf{%
20} (2012), 310--327.

\bibitem{DoRo} A. L. DONTCHEV and R. T. ROCKAFELLAR, \emph{Implicit
Functions and Solution Mappings: A View from Variational Analysis}, 2nd
edition, Springer, New York, 2014.

\bibitem{GCPT19} M. J. GISBERT, M. J. C\'{A}NOVAS, J. PARRA and F. J.
TOLEDO, \emph{Lipschitz modulus of the optimal value in linear programming},
J. Optim. Theory Appl. \textbf{182} (2019), 133--152.

\bibitem{GoLo98} M. A. GOBERNA and M. A. L\'{O}PEZ, \emph{Linear
Semi-Infinite Optimization}, John Wiley \& Sons, Chichester, UK, 1998.

\bibitem{huy-mor-yao08} N. Q. HUY, B. S. MORDUKHOVICH and J. C. YAO, \emph{%
Coderivatives of frontier and solution maps in parametric multiobjective
optimization}, Taiwanese J. Math. \textbf{12} (2008), 2083--2111.

\bibitem{Ioffe17} A. D. IOFFE, \emph{Variational Analysis of Regular Mappings%
}, Springer, Cham, Switzerland, 2017.

\bibitem{KlKu02} D. KLATTE and B. KUMMER, \emph{Nonsmooth Equations in
Optimization: Regularity, Calculus, Methods and Applications}, Kluwer
Academic, Dordrecht, The Netherlands, 2002.

\bibitem{Mo93} B. S. MORDUKHOVICH, \emph{Complete characterizations of
openness, metric regularity, and Lipschitzian properties of multifunctions},
Trans. Amer. Math. Soc. \textbf{340} (1993), 1--35.

\bibitem{Mor06i} B. S. MORDUKHOVICH,\emph{Variational Analysis and
Generalized Differentiation, I: Basic Theory, II: Applications}, Springer,
Berlin, 2006.

\bibitem{Mor18} B. S. MORDUKHOVICH,\emph{Variational Analysis and
Applications}, Springer, Cham, Switzerland, 2018.

\bibitem{Robinson} S. M. ROBINSON, \emph{Some continuity properties of
polyhedral multifunctions}, Mathematical Programming at Oberwolfach (Proc.
Conf., Math. Forschungsinstitut, Oberwolfach, 1979). Math. Programming Stud.
No. \textbf{14} (1981), 206--214.

\bibitem{rw} R. T. ROCKAFELLAR and R. J-B. WETS, \emph{Variational Analysis}%
, Springer, Berlin, 1998.
\end{thebibliography}
\end{document}